\documentclass[12pt]{amsart}
\usepackage{amsmath,amssymb,amsfonts}
\usepackage{mathtools}
\usepackage[shortlabels]{enumitem}
\setlist{leftmargin=*}

\makeatletter
\def\@tocline#1#2#3#4#5#6#7{\relax
  \ifnum #1>\c@tocdepth 
  \else
    \par \addpenalty\@secpenalty\addvspace{#2}%
    \begingroup \hyphenpenalty\@M
    \@ifempty{#4}{%
      \@tempdima\csname r@tocindent\number#1\endcsname\relax
    }{%
      \@tempdima#4\relax
    }%
    \parindent\z@ \leftskip#3\relax \advance\leftskip\@tempdima\relax
    \rightskip\@pnumwidth plus4em \parfillskip-\@pnumwidth
    #5\leavevmode\hskip-\@tempdima
      \ifcase #1
       \or\or \hskip 1em \or \hskip 2em \else \hskip 3em \fi%
      #6\nobreak\relax
    \dotfill\hbox to\@pnumwidth{\@tocpagenum{#7}}\par
    \nobreak
    \endgroup
  \fi}
\makeatother
\usepackage[sorted]{amsrefs}

\newtheorem{thm}{Theorem}[section] 
\newtheorem{cor}[thm]{Corollary}
\newtheorem{lem}[thm]{Lemma} 
\newtheorem{prop}[thm]{Proposition}
\newtheorem{claim}[thm]{Claim} 
\newtheorem{fact}[thm]{Fact}
\theoremstyle{definition} 
\newtheorem{defn}[thm]{Definition}
\newtheorem{rem}[thm]{Remark}
\newtheorem{sample}[thm]{Example} 
\numberwithin{equation}{section}

\newcommand{\RR}{\mathbb{R}}
\newcommand{\CC}{\mathbb{C}}

\gdef\CF{\mathcal F}
\newcommand{\ZZ}{\mathbb{Z}}
\newcommand{\NN}{\mathbb N}
\newcommand{\la}{\langle}
\newcommand{\ra}{\rangle}
\newcommand{\PP}{\mathbb{P}} 
\gdef\tp{{\mathrm{tp}}}
\newcommand{\sub}{\subseteq}
\def\CL{\mathcal{L}}
\def\st{\operatorname{st}}
\def\cl{\operatorname{cl}}

\def\res{\mathrm{res}}

\def\CA{\mathcal A}
\def\CO{\mathcal O}
\def\full{{\mathrm{full}}}
\def\FR{\mathfrak{R}}
\def\FC{\mathfrak{C}}
\gdef\CX{\mathcal{X}}
\gdef\CY{\mathcal{Y}}
\gdef\CP{\mathcal{P}}
\gdef\ri{\mathrm{i}}
\gdef\Stab{{\mathrm{Stab}}}
\gdef\val{{\mathrm{val}}}
\gdef\sa{{\mathrm{sa}}}
\gdef\om{{\mathrm{om}}}
\gdef\al{{\mathrm{a}}}
\gdef\tame{{\mathrm{tame}}}

\gdef\Graff{{\mathrm{Graff}}}
\gdef\RRL{\mathbb{R}\mathrm{L}}
\gdef\CCL{\mathbb{C}\mathrm{L}}

\usepackage[colorlinks=true, linkcolor=blue]{hyperref}

\begin{document}

\title[flows on tori]{Algebraic and o-minimal flows on complex and real tori}

\author{Ya'acov Peterzil}
\address{University of Haifa}
\email{kobi@math.haifa.ac.il}
\author{Sergei Starchenko}
\address{University of Notre Dame}
\email{sstarche@nd.edu}
\thanks{Both authors thank the Israel-US Binational Science Foundation for its
  support. The second author was supported by the NSF research grant DMS-1500671}

\begin{abstract}
We consider the covering map $\pi:{\mathbb C}^n\to \mathbb T$ of a compact complex torus.
Given an algebraic variety $X\subseteq {\mathbb C}^n$ we describe the topological closure of
$\pi(X)$ in $\mathbb T$. We obtain a similar description when $\mathbb T$ is a real
torus and $X\subseteq {\mathbb R}^n$ is a set definable in an o-minimal structure over the reals.
\end{abstract}

 \maketitle

\vspace{-10pt}

\tableofcontents

\section{Introduction}
\label{sec:introduction}

 Let $A$ be a complex abelian variety of dimension $n$,  and let $\pi:\CC^n \to A$ be its covering map.
It follows from  a theorem of Ax (see \cite [Theorem 3]{ax1}), that if $X\sub \CC^n$
is an algebraic variety then the Zariski closure of $\pi(X)$ is a union  of finitely
many cosets of abelian subvarieties of $A$.

In \cites{UY, flow}, Ullmo and Yafaev attempt to characterize the \emph{topological} 
closure of $\pi(X)$ in the above setting and also when $X$ is a set
definable in an o-minimal expansion of the real field.

 They prove a similar result to Ax's for algebraic curves (see \cite[Theorem 2.4]{UY}:
 If $X\sub \CC^n$ is an irreducible algebraic curve then the topological closure of
$\pi(X)$ in $A$ is
 \[\cl(\pi(X))=\pi(X)\cup \bigcup_{k=1}^m Z_k,\]
where each $Z_k$ is a real weakly
 special subvariety of $A$, namely a coset of a real Lie subgroup of $A$.
 They conjecture that the same is true  for algebraic
 subvarieties  $X\sub \CC^n$ of arbitrary dimension.

 In this article we give a full description of $\cl(\pi(X))$ when $X$ is an  algebraic
 subvariety of $\CC^n$  of arbitrary dimension and also when $X\sub \RR^n$ is
 definable in an o-minimal structure over the reals and
 $\pi:\mathbb R^n\to \mathbb T$ is the covering map of a compact
 real torus.

As  we
 show, the conjecture from \cite{UY} fails as stated (see Section~\ref{sec:examples})
and we prove a modified version by showing that the frontier of
$\pi(X)$ consists of finitely many  families of real weakly
special subvarieties. Our theorem holds for arbitrary compact
complex tori and not only for abelian varieties.

 \begin{thm}\label{intro:thm1} Let $\pi:\CC^n\to \mathbb T$ be the covering
 map of a compact complex torus and
  let $X$ be an algebraic subvariety of $\CC^n$.
 Then there are finitely  many algebraic subvarieties $C_1,\ldots, C_m\sub
 \CC^n$ and finitely many real subtori (i.e real Lie subgroups)
 $\mathbb{T}_1,\ldots, \mathbb{T}_m\sub \mathbb T$ of positive
 dimension such that
\[\cl(\pi(X))=\pi(X)\cup \bigcup_{i=1}^m (\pi(C_i)+\mathbb{T}_i).\]

In addition,
\begin{enumerate}[(i)]

\item For every $i=1,\ldots, m$, we have $\dim_\CC C_i<\dim_\CC
X$.

\item   If $\mathbb{T}_i$ is maximal with respect to inclusion
among the subtori then $C_i$ is finite.
\end{enumerate}
\end{thm}

  Notice that in general the sets
$\pi(X)$ and $\pi(C_i)$ need neither be closed nor definable in any o-minimal
structure. Note also that when $\dim_\CC X{=}1$ then $\dim_\CC C_i {=}0$
hence  Theorem~\ref{intro:thm1} implies
the result of Ullmo and Yafaev mentioned above.

In fact, as we show, the choice of $C_i$ depends only on $X$ (and not on
$\mathbb T$) and furthermore, each subtorus $\mathbb{T}_i\sub \mathbb T$ in the
above description is of the form $\cl(\pi(V_i))$ with $V_i\sub \CC^n$  a complex
linear subspace which also depends only on $X$.
\\

In order to prove the above result, we find it more convenient to
work in $\CC^n$ rather than in $\mathbb T$.  Let $\Lambda=\ker
\pi$ be the corresponding lattice in $\CC^n$ and let
$\cl(X+\Lambda)$ denote the topological closure in $\CC^n$. It is
easy to see that $\cl(\pi(X))=\pi(\cl(X+\Lambda))$ hence the above
theorem can be deduced from our analysis of $\cl(X+\Lambda)$ in
$\CC^n$, which we now describe.

For $V$  a complex or real linear subspace  of $\CC^n$ and $\Lambda$ a lattice
 in $\CC^n$ we denote by $V^\Lambda$ the smallest $\RR$-linear subspace of
$\CC^n$ containing $\Lambda$ with a basis in $\Lambda$ (equivalently, this is the
connected component of $0$ in the real Lie group $\cl(V+\Lambda)$).

The following is the main result of the first part of this article.
\\

\noindent{\bf Main Theorem (algebraic case)}[see
Theorem~\ref{thm:main0}]. \emph{Let $X\subseteq \CC^n$ be  an
algebraic subvariety. There are linear $\CC$-subspaces
$V_1,\dotsc,V_m \subseteq
 \CC^n$ of positive dimension  and algebraic subvarieties $C_1, \dotsc,C_m \subseteq \CC^n$
 such that for any lattice
 $\Lambda<\CC^n$ we have
\[  \cl(X+\Lambda) =(x+\Lambda)
    \cup\bigcup_{i=1}^m (C_i+V_i^\Lambda+\Lambda).
\]
In addition,
\begin{enumerate}[(i)]
\item For each $i=1,\dotsc, m$, we have $\dim_\CC C_i<\dim_\CC X$.
\item For each $V_i$ that is maximal among $V_1,\dotsc, V_m$, the
set $C_i\sub \CC^n$ is  finite.
\end{enumerate}
}

\medskip

Notice that Theorem~\ref{intro:thm1} is an immediate corollary of the above.

\medskip

 In the second part of the article we obtain a similar result in the
o-minimal setting, and disprove the analogous o-minimal conjecture from \cite{flow}.
In order to formulate the theorem, we fix an o-minimal expansion $\RR_{\om}$ of the
real field.

\begin{thm}\label{thm2} Let $\pi:\RR^n\to \mathbb T$ be the covering map of a compact real torus and
let $X\sub \RR^n$ be a closed set definable in $\RR_{\om}$. There
are finitely many definable closed sets $C_1,\ldots, C_m\sub
\RR^n$ and finitely many real subtori $\mathbb{T}_1,\ldots,
\mathbb{T}_m\sub \mathbb \mathbb{T}$ of positive dimension such
that
\[\cl(\pi(X))=\pi(X)\cup \bigcup_{i=1}^m (\pi(C_i)+\mathbb{T}_i).\]

In addition,
\begin{enumerate}[(i)]
\item  For every $i=1,\ldots, m$, we have $\dim_\RR C_i<\dim_\RR
X$ (where $\dim_\RR$ is the o-minimal dimension). \item  If
$\mathbb{T}_i$ is maximal with respect to inclusion among
$\mathbb{T}_1,\dotsc,\mathbb{T}_m$ then $C_i$ is bounded in
$\CC^n$ and in particular $\pi(C_i)$ is closed.
\end{enumerate}

\end{thm}

As in the algebraic case, the above result follows from a theorem on the closure of
$X+\Lambda$ in $\RR^n$, for $\Lambda= \ker \pi$.
\\

\noindent {\bf Main Theorem (o-minimal case)} [see
Theorem~\ref{thm:main1-omin}]. {\it Let $X\subseteq \RR^n$ be a
  closed set definable in an o-minimal expansion $\RR_{\om}$ of the real field.
   There are linear $\RR$-subspaces  $V_1,\dotsc,V_m \sub
 \RR^n$ of positive dimension, and for each $i=1,\dotsc, m$  definable closed
 $C_i\sub \RR^n$,  such that for any lattice
 $\Lambda<\RR^n$ we have
\[  \cl(X+\Lambda) =(X+\Lambda)
    \cup\bigcup_{i=1}^m (C_i+V_i^\Lambda+\Lambda).
\]

In addition,
\begin{enumerate}[(i)]
\item For each $i=1,\ldots, m$, we have $\dim_\RR C_i<\dim_\RR X$.
\item For each $V_i$ that is maximal among $V_1,\ldots, V_m$, the
set $C_i\sub \RR^n$ is  bounded, and in particular
$C_i+V_i^{\Lambda}+\Lambda$ is a closed set.
\end{enumerate}
}

\medskip

 The proofs of the  above theorems are carried out in two main steps. In the
 first step we describe the closure of $X+\Lambda$ in $\CC^n$ and $\RR^n$, as a union of closures
 of sets of the form $\Lambda+A$, where $A$ is an affine subspace
 of $\CC^n$ or $\RR^n$. We call these affine spaces
 ``affine asymptotes'' to $X$ (see Section~\ref{sec:affine-assymptotes}). The analysis of the closure in terms of
affine asymptotes uses the model theoretic notion of types. We
also apply at this step the theory of stabilizers of $\mu$-types
as was developed in \cite{mustab} (see
Section~\ref{sec:mu-stabilizers-types} below). Furthermore, using
model theory of valued fields
 and of o-minimal structures we show that the family of affine
 asymptotes to $X$ is itself constructible (in the algebraic case)
 and definable (in the o-minimal case). We introduce these model theoretic preliminaries in Section~\ref{sec:universal-domain} and
Section~\ref{sec:valu-field-struct}.

In the second step we use Baire Category Theorem to replace the infinitely many
affine spaces by finitely many (each possibly infinite) families of translates of
fixed linear spaces thus yielding Theorem~\ref{thm:main0} and
Theorem~\ref{thm:main1-omin}.
\\

We note that in the same articles, Ullmo and Yafaev  formulate two
measure theoretic conjectures about $\pi(X)$, in the algebraic and
o-minimal settings, and we do not touch on them here.
\\

Finally, although the article is not formulated in that language,
our approach is influenced   by  van den Dries work on various
notions of limits of definable families and their connection to
model theory (see \cite{lou-limit}).

\section{Preliminaries}
\label{sec:universal-domain}

\subsection{Model theoretic preliminaries}
We first introduce the model theoretic settings in which we will be working.
We refer to \cite{marker}  for basics on model theory and to
\cite{omin} and \cite{CM}  for basics on o-minimality.

\medskip

We denote by $\CL_\al$ the ``algebraic'' language, i.e.\ the language of
rings $\CL_\al=\la +,-,\cdot,0,1\ra$,  and view the field
$\CC$ as an $\CL_\al$-structure.

Working with the field  $\RR$ and semialgebraic sets we use
the ``semialgebraic'' language  $\CL_\sa=\la +,-,\cdot,<,0,1\ra$.

For the algebraic case we also need a language for valued field.
We use the language $\CL_\val=\la +,-,\cdot, \CO,0,1\ra$, where $\CO$
is a unary predicate  with an intended use for the valuation
ring. Notice that $\CL_\al \subseteq \CL_\val$.

For the o-minimal case
 we  fix  an o-minimal expansion $\RR_\om$ of
the field  $\RR$ and denote its language by $\CL_\om$. Notice that $\CL_\al\sub
\CL_\sa\subseteq \CL_\om$.

We also  work in expansions  of  $\RR$ and $\RR_\om$
by various additive subgroups $\Lambda\leq \RR^n$. To avoid
different expansions it is
 convenient to treat them all at once. Thus we consider the expansion of $\RR$
by predicates for {\bf all} subsets of $\RR^n$, for all $n$. The language for this structure is
denoted by $\CL_\full$.
We let $\RR_\full$ be the associated $\CL_\full$-structure on
$\RR$.

We  choose  a cardinal $\kappa > 2^\omega$ and fix a $\kappa$-saturated elementary
extension $\FR_\full$ of $\RR_\mathrm{full}$.  We denote by $\FR$ the
underlying real closed field and by $\FR_\om$ the o-minimal reduct $\FR_\full
{\restriction} \CL_\om$. Notice that since both the real closed field $\FR$ and the
o-minimal structure $\FR_\om$ are reducts of $\FR_\full$ they are both
$\kappa$-saturated.

To distinguish between subsets of $\FR$ and $\RR$ we use  the following
convention: we let roman lettes $X,Y,Z$ etc.\  denote subsets of $\RR^n$ and script
letters $\mathcal{X,Y,Z}$ etc.\  subsets of $\FR^n$.

Also if $X\subseteq \RR^n$,  then we can view $X$ as an $\CL_\full$-definable set
and denote by $X^\sharp$  the set $X(\FR)$ of realizations of $X$ in
$\FR_\full$.

\medskip

Our model theoretic terminology is standard.
By definable we mean definable with parameters.
 We use  ``$\CL_\sa$-definable'' and  ``semialgebraic'' interchangeably.

If $\CL_\bullet$ is one of our languages and $x$ a finite tuple of
variables  then an   $\CL_\bullet$-type $p(x)$ (over
a set $A\subseteq \FR$) is a collection of $\CL_\bullet$-formulas (over $A$)  with free variables
$x$. We  identify an $\CL_\bullet$-type $p(x)$ with the collection of subsets of
$\FR^{|x|}$ defined by formulas in $p(x)$. We do not assume that a type is complete,
but always assume it is consistent.   Thus  an  $\CL_\bullet$-type $p(x)$ over $A$
is a collection of subsets of $\FR^{|x|}$ such that  each subset is
$\CL_\bullet$-definable over $A$ and $p(x)$ satisfies the finite intersection
property (namely the intersection of a every finite subcollection of $p$ is nonempty).
Given a type $p(x)$ we denote by $p(\FR)$ the set $\bigcap_{\CX\in p} \CX$.   Two
$\CL_\bullet$-types $p(x)$ and $q(x)$ are called equivalent if for every finite
$p_0(x)\subseteq p(x)$ there is finite $q_0(x)\subseteq q(x)$ with $q_0(\FR)
\subseteq p_0(\FR)$, and vise versa. It follows that $p(\FR)=q(\FR)$.

\begin{defn}
A subset $\CX\subseteq \FR^n$ is called \emph{pro-semialgebraic} (over $A\sub \FR$)
if there is $A_0\subseteq \FR$ ($A_0\subseteq A$) with $|A_0|< \kappa$ and a
semialgebraic type $p(x)$ over $A_0$ such that  $\CX = p(\FR)$.

\end{defn}
 Notice that if $\CX$ and $p(x)$ are as in the above definition and  $\CX=q(\FR)$ 
for  another semialgebraic type $q(x)$ over  $A_0$
 then  by $\kappa$-saturation of $\FR$ the types $p(x)$ and
$q(x)$ are equivalent.

\subsection{Basics on additive subgroups}
\label{sec:basics-subgroups-rrn}

Let $W$ be a finite dimensional $\RR$-vector space and $\Lambda$ be an
additive subgroup of
$W$ whose $\RR$-span is the whole $W$.
We do not assume that $\Lambda$ is a lattice or even finitely
generated.

We say that an $\RR$-subspace $V\subseteq W$ is \emph{defined over $\Lambda$} if it
has a basis consisting of elements of $\Lambda$.

It is not hard to see that the family of subspaces defined over $\Lambda$ is closed
under arbitrary intersections and finite sums.

For a subspace $H\subseteq W$  we  denote by $H^\Lambda$ the smallest
$\RR$-subspace of $W$ defined over $\Lambda$ containing $H$.

We will need the following well-known fact.

\begin{fact}\label{fact:kron}
Let $W$ be a finite dimensional $\RR$-vector space and $\Lambda\leq W$ an additive
subgroup whose $\RR$-span is the whole $W$. If  $H\subseteq W$ is an  $\RR$-subspace
then $H^\Lambda+\Lambda\subseteq\cl(H+\Lambda)$ (with equality when $\Lambda$
is a lattice in $W$).
\end{fact}

\begin{rem} For  a  $\CC$-subspace  $H\subseteq \CC^n$ and an additive
  subgroup $\Lambda\sub \CC^n$ whose
$\RR$-span is the whole $\CC^n$, we still  denote by $H^\Lambda$ the smallest $\RR$-subspace
of $\CC^n$ containing $H$  and having an $\RR$-basis in
  $\Lambda$. Thus  $H^{\Lambda}$ need not be a $\CC$-linear
  subspace of $\CC^n$.
\end{rem}

\section{Valued field structures on $\FR$}
\label{sec:valu-field-struct}

We denote by $\CO_\FR$ the convex hull of $\RR$ in $\FR$. It is a valuation ring of
$\FR$, and  we let $\mu_\FR$ be its maximal ideal. The set $\mu_\FR$ is the
intersection of all open intervals $(-1/n,1/n)$ for $n\in \NN^{>0}$, hence it is
pro-semialgebraic over $\varnothing$.

\medskip

As an additive group $\CO_\FR$  is the direct sum
$\CO_\FR=\RR\oplus\mu_\FR$, hence for $\alpha\in \CO_\FR$ there is
unique $r_\alpha\in \RR$ with $\alpha\in r_\alpha+\mu_\FR$.  This
$r_\alpha$ is called \emph{the standard part of $\alpha$} and  we
denote it by  $\st(\alpha)$. Thus we have \emph{the standard part map}
$\st\colon \CO_\FR\to \RR$.  Slightly
abusing notations,  we use $\st(x)$ also to denote the map from
$\CO_\FR^n$ to $\RR^n$ defined by $\st(x_1,\dotsc,x_n)=(\st(x_1),\dotsc,\st(x_n))$, and
for a subset $\CX\subseteq \FR^n$ we write $\st(\mathcal X)$ for the set
 $\st(\mathcal X\cap \CO_\FR^n)$.

\subsection{Closure and the standard part map}
\label{sec:clos-stand-part}

We need the following claim that relates the topological closure and the standard
part map. It follows from the saturation assumption on $\FR_\full$. As usual for
$\CX,\CY \subseteq \FR^n$ we write  $\CX+\CY$ for the set $\{ x+y \colon x\in \CX,
y\in \CY \}$. Recall that for a subset  $X \subseteq \RR^n$ we denote
by $X^\sharp$  the subset of $\FR^n$ defined in the structure
$\FR_\full$ by the predicate  corresponding to $X$. 

\begin{claim}\label{claim:cl-stand-part} \begin{enumerate}
\item For $X \subseteq \RR^n$
  we have  $\cl(X)= \st(X^\sharp)$. In particular, for $X, Y\sub \RR^n$ we have
  $\cl(X+Y)=\st(X^\sharp+Y^\sharp)$.
  \item  Let $\Sigma$ be a collection of subsets of $\RR^n$.
  Then
\[\st\left( \bigcap_{X\in \Sigma} X^\sharp\right)=\bigcap_{X\in \Sigma} \st(X^\sharp).\]
In particular the
set $\st\left( \bigcap_{X\in \Sigma} X^\sharp\right) $ is closed.
\end{enumerate}
\end{claim}

\subsection{The algebraic closure $\FC$ of $\FR$ as an ACVF structure.}
\label{sec:acvf-structure-uu}\leavevmode

Let $\FC=\FR+\ri\FR$, where $\ri=\sqrt{-1}$.   It is an algebraically closed field
containing $\CC$.  We identify the underlying set of $\CC$ with $\RR^2$ and
 the underlying set of  $\FC$ with $\FR^2$. We also view $\RR$ and
 $\FR$ as 
subfields of $\CC$ and  $\FC$, respectively, in an obvious way.

 Let $\CO_\FC \subseteq \FC$ be the set $\CO_\FC=\CO_\FR+\ri\CO_\FR$.  It is a valuation ring of $\FC$ with the
 maximal ideal $\mu_\FC =\mu_\FR+\ri\mu_\FR$.  Again we have that
 $\CO_\FC=\CC\oplus\mu_\FC$, and we let
 $\st\colon \CO_\FC\to \CC$ be the standard part map.

 \begin{rem}\label{rem:ACVF}
We can also identify $\CC$ with the residue
 field $k=\CO_\FC/\mu_\FC$ so that  the residue map
 $\res\colon \CO_\FC\to k$ is the same as the standard part map $\st\colon
 \CO_\FC\to \CC$.
 \end{rem}

We denote by $\FC_\val$ the $\CL_\val$-structure $(\FC;  +, -,\cdot,  \CO_\FC, 0,
1)$.

\begin{prop}\label{prop:ACVF}

If $\CX\subseteq \FC^n$ is an $\CL_\val$-definable set
then $\CX\cap \CC^n$ is $\CL_\al$-definable in the field $\CC$, i.e.\ it is
a constructible  subset of $\CC^n$.

If $\CF$ is an $\CL_\val$-definable family of subsets of $\FC^n$
then the family $\{ \mathfrak{F}\cap \CC^n \colon \mathfrak{F}\in
\CF\}$  is constructible.
\end{prop}
\begin{proof}
 By Remark~\ref{rem:ACVF}, the field $\FC$ together with a predicate for $\CC$ and
  the map $\st(x)$ is an algebraically closed field with an embedded
  residue field. By \cite[Lemma 6.3]{HK} the embedded residue field $\CC$  is
  stably embedded, i.e.\ for every
  $\CL_\val$-definable subset
  $\CX\subseteq \FC^n$  the set  $\CX\cap \CC^n$
is  $\CL_\al$-definable in the field of complex numbers.

The second  part of the proposition is not stated in \cite[Lemma 6.3]{HK}, but it
easily follows: By quantifier elimination from \cite[Lemma 6.3]{HK}, the theory of
algebraically closed valued fields of characteristic zero with an embedded residue
field is complete, and we can use a standard compactness argument.
\end{proof}

\begin{rem}
Notice that the structure $\FC_\val$ is not a reduct of 
$\FR_\full$ (e.g.  $\CO_\FC$ is not definable in $\FR_\full$),  so 
$\FC_\val$ need not be $\kappa$-saturated, and in fact it is not. 
For example the set 
$\{ x\in \CO_\FC \wedge x\notin (c+\mu_\FC) \colon c\in \CC\}$
is an $\CL_\val$-type  over $\CC$  but has no 
realization in $\FC$.

However, as we will see below (see Corollary~\ref{cor:acvf-semi}), $\CL_\val$-types over $\CC$ that are realized in
$\FC$ can be viewed as pro-semialgebraic objects, and working with them we  will
make use of the saturation of the field $\FR$.
\end{rem}

 Using  the
identification of $\FC^n$ with $\FR^{2n}$, we say that a subset $\CX \subseteq
\FC^n$ is semialgebraic (over $A\subseteq \FR$)  if it is semialgebraic (over $A$)
as a subset of $\FR^{2n}$. Similarly we say that a set $\CX\subseteq \FC^n$ is
pro-semialgebraic if it is pro-semialgebraic as a subset of $\FR^{2n}$.

For example, every constructible subset of $\FC^n$ is also
semialgebraic, and for every $n\in \NN$ the set $\mu_\FC^n$ is
pro-semialgebraic over $\emptyset$.  However the set $\CO_\FC$ is not
pro-semialgebraic.

Given an element $\alpha\in \FC^n$ we will consider its $\CL_\val$-type over $\CC$
and also its semialgebraic type over $\RR$.
 To distinguish these types we denote by
$\tp_\val(\alpha/\CC)$ the complete $\CL_\val$-type of $\alpha$ over
$\CC$, i.e.
\[ \tp_\val(\alpha/\CC)=\{ \CX\subseteq \FC^n  \colon \alpha\in\CX, \  \CX\text{ is
    $\CL_\val$-definable over $\CC$ } \}, \]
and by $\tp_\sa(\alpha/\RR)$  the  semialgebraic  type of $\alpha$ over $\RR$,
i.e.
\[ \tp_\sa(\alpha/\RR)=\{ \CX  \subseteq \FC^n \colon
  \alpha\in\CX, \  \CX\text{ is
    semialgebraic over $\RR$ } \}. \]
Notice that $\tp_\sa(\alpha/\RR)$ can be also written as
\[ \{ X^\sharp \subseteq \FC^n
  \colon \alpha\in X^\sharp,  \  X \subseteq \CC^n \text{ is
    semialgebraic } \}. \]

The following theorem plays an essential role in this paper.

\begin{thm}\label{thm:acvf-semi}
Let $\alpha\in \FC$ and $p(x)=\tp_\val(\alpha/\CC)$.
There is an $\CL_\val$-type $s(x)$ over $\CC$ that is equivalent to  $p(x)$
and such that every $\CX\in s(x)$ is pro-semialgebraic over $\RR$.
\end{thm}
\begin{proof}
By Robinson's quantifier elimination (see \cite{rob}),  if
$\CX\subseteq \FC^n$ is  $\CL_\val$-definable over $\CC$ then it  is a finite Boolean combination
of sets of the form $X^\sharp$  where $X\subseteq \CC^n$ is a
constructible set, and sets of the form $\{ z\in \FC^n \colon h(z)\in
q(z)\CO_\FC \}$  where $h,q\in \CC[x]$.

Since the complement of a constructible set is constructible, and, for
$h,q\in \CC[x]$,
the complement of the set $\{ z\in \FC^n \colon h(z)\in
q(z)\CO_\FC \}$ is
$\{ z\in \FC^n \colon h(z)\neq 0, q(z)/h(z)\in \mu_\FC
   \}$,
every$\CL_\val$-definable over $\CC$ set  $\CX\subseteq \FC^n$
is a finite \textbf{positive} Boolean combination of sets of
the following three   kinds:
 \begin{enumerate}[1.]
 \item $X^\sharp$, where $X\subseteq \CC^n$ is a constructible
   set.
 \item $\{ z\in \FC^n \colon h(z)\neq 0,\, q(z)/h(z)\in \mu_\FC
   \}$, where $h,q\in \CC[x]$.
 \item $\{ z\in \FC^n \colon  h(z) \in q(z) \CO_\FC
   \}$, where $h,q\in \CC[x]$.
 \end{enumerate}
Notice that sets of all three  kinds are $\CL_\val$-definable over $\CC$. Every set
of the first kind is also a semialgebraic set defined over $\RR$, and every set of
the second kind is pro-semialgebraic over $\RR$ (since $\mu_\FC$ is
pro-semialgebraic).

Let $\CX\in p(x)$ be of the third kind, i.e.
\[ \CX=\{ z\in \FC^n \colon  h(z)\in q(z)\CO_\FC
   \}.\]
Since $\alpha\in \CX$, we have $h(\alpha)\in q(\alpha)\CO_\FC$.
Then either $q(\alpha)=0$ (and hence $h(\alpha)=0$) or, for $c=\st( h(\alpha)/q(\alpha))$, we have
   $h(\alpha)/q(\alpha) -c \in \mu_\FC$.

In either case  we  get a set $\CY$ of the first or second kind with
$\alpha\in \CY$ and $\CY \subseteq \CX$.

Thus if we take $s(x)$ to be the set of all $\CX\in p$ of first and
second kinds, then $s(x)$ is equivalent to $p(x)$ and consists of
sets that are pro-semialgebraic over $\RR$.
\end{proof}

\begin{cor}\label{cor:acvf-semi}
Let $\alpha\in \FC^n$ and $p(x)=\CL_\val(\alpha/\CC)$. There is an $\CL_\sa$-type
$r_\alpha(x)$ over $\RR$  such that
\begin{enumerate}
\item  $r_\alpha(\FC)=p(\FC)$.
\item For every finite $r'(x)\subseteq r_\alpha(x)$ there is $\CX\in
  p(x)$ with $\CX \subseteq r'(x)$.
\end{enumerate}

\end{cor}

\begin{rem}\label{rem:val-sm}
Notice that unless $\alpha\in \CC^n$ the semialgebraic
  type  $r_\alpha(x)$ is different from the semialgebraic type
$p_\sa(x)=\tp_\sa(\alpha/\RR)$ and we only have strict inclusion  $p_\sa(\FC)\subset
  r_\alpha(\FC)=p(\FC)$.
 \end{rem}

Using the  $\kappa$-saturation of the field
$\FR$ and Corollary~\ref{cor:acvf-semi} we obtain the following.

\begin{cor}\label{cor:acvf-sat} Let $\alpha\in \FC^n$, $p(x)=\tp_\val(\alpha/\CC)$ and
$\mathcal{Z}\subseteq \FC^n$ a pro-semialgebraic set.
If $\mathcal{Z}\cap \mathcal{X}\neq \emptyset$ for every
$\mathcal{X}\in p(x)$ then $\mathcal{Z}\cap p(\FC)\neq \emptyset$.
\end{cor}

\begin{cor}\label{cor:acvf-proj} Let $h(x)$ be a polynomial over
  $\CC$,  $\alpha\in \FC^n$, $\alpha_1=h(\alpha)$, $p(x)=\tp_\val(\alpha/\CC)$ and
 $p_1(x)=\tp_\val(\alpha_1/\CC)$.  Then $h(x)$ maps $p(\FC)$  onto
$p_1(\FC)$.
\end{cor}
\begin{proof}
 Let $s(x)$ be an $\CL_\val$-type over $\CC$ equivalent
  to $p(x)$ consisting of pro-semialgebraic sets, as in Theorem~\ref{thm:acvf-semi}.

First notice that $p_1(x)$ is equivalent to the $\CL_\val$ type
\[\{ h(p'(\FC))\colon p'(x)\subseteq p(x) \text{ is finite}\}.\]

Let $s_1(x)=\{ h(s'(\FC))\colon s'(x)\subseteq s(x) \text{ is
  finite}\}$.
 It is easy to see that $s_1(x)$ consists of pro-semialgebraic
sets and since $s(x)$ and $p(x)$
are equivalent,
$s_1(x)$ and  $p_1(x)$ are  equivalent as well.

As $h(x)$ is a polynomial over $\CC$ it is also a semialgebraic map, and by the
$\kappa$-saturation of $\FR$ we  have $h(s(\FC))= s_1(\FC)$. Since $s(\FC)=p(\FC)$
and $s_1(\FC)=p_1(\FC)$ the result follows.
\end{proof}

\subsection{On $\mu$-stabilizers of types}\label{sec:mu-stabilizers-types}

\subsubsection{The o-minimal case}
\label{sec:o-minimal-case}\leavevmode

We review briefly $\mu$-stabilizers of $\CL_\om$-types
over $\RR$  and refer to  \cite{mustab} for more details.

Since the structure $\FR_\om$ is $\kappa$-saturated the following
definition is equivalent to  \cite[Definition 2.10]{mustab}.

\begin{defn}\label{defn:o-min-stab}
For $\alpha\in \FR^n$ and $p(x)=\tp_\om(\alpha/\RR)$ we define \emph{the 
$\mu$-stabilizer of $p$} as 
\begin{equation*}
 \Stab^\mu_\om(p)=
\{ v\in \RR^n \colon v+(p(\FR) +\mu^n_\FR) = (p(\FR)+\mu^n_\FR) \}, 
\end{equation*}
and we also denote it by $\Stab^\mu_\om(\alpha/\RR)$.  
\end{defn}

The fact below follows from the  main results of \cite{mustab} (see Proposition 2.17
and Theorem 2.10 there).

\begin{fact}\label{fact:mustab}
Let $\alpha\in \FR^n$.
  \begin{enumerate}
  \item $\Stab^\mu_\om(\alpha/\RR)$ is an $\CL_\om$-definable subgroup of $(\RR^n,+)$.
  \item If $\alpha$ is unbounded, i.e. $\alpha\notin \CO_\FR^n$,  then
    $\Stab^\mu_\om(\alpha/\RR)$ is infinite.
  \end{enumerate}
\end{fact}

\subsubsection{The algebraic case}\leavevmode
\label{sec:algebraic-case}

Similarly to the o-minimal case we now define $\mu$-stabilizers for
$\CL_\val$-types over $\CC$ realized in $\FC$.

\begin{defn}\label{defn:val-stab}
For $\alpha\in \FC^n$ and $p(x)=\tp_\val(\alpha/\CC)$ we define \emph{the 
$\mu$-stabilizer of $p$} as 
\begin{equation*}
 \Stab^\mu_\val(p)=
\{ v\in \CC^n \colon v+(p(\FC) +\mu^n_\FC) = p(\FC)+\mu^n_\FC \}, 
\end{equation*}
and we also denote it by $\Stab^\mu_\val(\alpha/\CC)$.  

We  denote by  $\CP_\alpha^\mu$ the set
$\CP_\alpha^\mu=p(\FC)+\mu_n^\FC$. 
Thus   $\Stab^\mu_\val(\alpha/\CC)=\{ v\in \CC^n \colon
v+\CP_\alpha^\mu=\CP_\alpha^\mu\}$. 

\end{defn}

Since the structure $\FC_\val$ is not $\kappa$-saturated 
we need
some preliminaries before we can 
prove an analogue of   Fact~\ref{fact:mustab}.

\begin{lem}\label{lem:mu-add}  Let $\alpha\in \FC^n$. For $v\in \CC^n$
  we have $v+\CP^\mu_\alpha=\CP^\mu_{v+\alpha}$.
 
\end{lem}
\begin{proof} Let $p(x)=\tp_\val(\alpha/\CC)$,  $\beta=v+\alpha$ and
  $q(x)=\tp(\beta/\CC)$.

  It is easy to see that $v+p(\FC)=q(\FC)$, hence
\[v+\CP_\alpha^\mu= v+p(\FC)+\mu_n^\FC =q(\FC)+\mu_n^\FC=P_\beta^\mu=P_{v+\alpha}^\mu.\]
\end{proof}

\begin{prop}\label{prop:mu-sat} For  $\alpha\in \FC^n$ we have
\[ \CP_\alpha^\mu
=
\cap \{\CX+\mu_\FC^n \colon \CX\subseteq \FC^n  \text{ is
  $\CL_\val$-definable over $\CC$ with } 
\alpha\in \CX+\mu_\FC^n\}. \]
\end{prop}
\begin{proof} Let $p(x)=\tp_\val(\alpha/\CC)$.

  The inclusion $\subseteq$ is easy. Indeed, let $\CX\subseteq \FC^n$ be
$\CL_\val$-definable over $\CC$ with $\alpha\in \CX+\mu_\FC^n$.  Then
$\CX+\mu_\FC^n$ is $\CL_\val$-definable over $\CC$ as well,  hence
$p(\FC)\subseteq \CX+\mu_\FC^n$ and 
\[\CP_\alpha^\mu=p(\FC)+\mu_\FC^n
\subseteq (\CX+\mu_\FC^n)+\mu_\FC^n= \CX+\mu_\FC^n.\]

For the inclusion $\supseteq$,  let 
\[ \beta\in \cap \{\CX+\mu_\FC^n \colon \CX\subseteq \FC^n  \text{ is
  $\CL_\val$-definable over $\CC$ with } 
\alpha\in \CX+\mu_\FC^n\}.\]
We need to show $\beta\in p(\FC)+\mu_\FC^n$, or equivalently
$(\beta+\mu_\FC^n) \cap p(\FC)\neq \emptyset$.  

Let $\CX\in p(x)$. Then
$(\CX+\mu_n^\FC) \in p(x)$ hence
 $\beta\in \CX+\mu_\FC^n$ and
$(\beta+\mu_\FC^n )\cap \CX\neq \emptyset$.   Since the set
$\beta+\mu_\FC^n$ is pro-semialgebraic, by
Corollary~\ref{cor:acvf-sat}, we obtain  $(\beta+\mu_\FC^n) \cap p(\FC)\neq \emptyset$.
\end{proof}

\begin{cor}\label{cor:mu-part}
For $\alpha,\beta\in \FC^n$ the following conditions are equivalent. 
\begin{enumerate}
\item  $\CP_\alpha^\mu=\CP^\mu_\beta$.
\item $\CP_\alpha^\mu \cap \CP^\mu_\beta \neq \emptyset$.
 \item $\alpha\in \CX+\mu_n^\FC \Leftrightarrow
\beta\in \CX+\mu_n^\FC$, for every $\CX\subseteq \FC^n$
that is $\CL_\val$-definable over $\CC$.
\end{enumerate}
\end{cor}
\begin{proof}
Obviously $(1) \Rightarrow (2)$. 

$(3) \Rightarrow  (1)$ follows from Proposition~\ref{prop:mu-sat}.

We are left to show that (2) implies (3).  We will assume (3) fails and
show that $\CP_\alpha^\mu\cap P_\beta^\mu=\emptyset$.   

 Assume (3) fails, and say  $\alpha\in \CX+\mu_n^\FC$,
but $\beta\notin \CX+\mu_n^\FC$ for some $\CX\subseteq \FC^n$
that is $\CL_\val$-definable over $\CC$.

Let $\CY=\FC^n\setminus
(\CX+\mu_n^\FC)$.  We have $\beta\in \CY$, the set 
$\CY$ is $\CL_\val$-definable over $\CC$ and $\CY+\mu_\FC^n=\CY$. 

By  Proposition~\ref{prop:mu-sat} we have $\CP_\alpha^\mu\subseteq
\CX+\mu_\FC^n$ and $
\CP_\beta^\mu\subseteq
\CY$. Since $(\CX+\mu^n_\FC)\cap \CY=\emptyset$ we get $\CP_\alpha^\mu \cap \CP^\mu_\beta = \emptyset$
\end{proof}

Thus the family of sets $\{\CP_\alpha^\mu \colon \alpha\in \FC^n\}$
partitions $\FC^n$, and, by Lemma~\ref{lem:mu-add}, translations by
elements of $\CC^n$ respect this partition.

We are ready to show that the $\mu$-stabilizers $\Stab^\mu_\val(\alpha/\CC)$
have properties analogous to $\mu$-stabilizers in o-minimal theories.

\begin{thm}\label{thm:mu-def}Let $\alpha\in \FC^n$.
  \begin{enumerate}
  \item    The  $\mu$-stabilizer $\Stab^\mu_\val(\alpha/\CC)$  is an
  algebraic subgroup of $(\CC^n,+)$, i.e.\ a $\CC$-subspace of
  $\CC^n$.
\item  If $\alpha\in \FC^n$ is unbounded, i.e. $\alpha\notin
  \CO_\FC^n$, then $\Stab^\mu_\val(\alpha/\CC)$ is infinite.
  \end{enumerate}

\end{thm}
\begin{proof}
(1).   By   Lemma~\ref{lem:mu-add} we have 
\[ \Stab^\mu_\val(\alpha/\CC)=\{ v\in \CC^n \colon \CP_\alpha^\mu=\CP_{v+\alpha}^\mu\}.\] 

Let $v\in \CC^n$. Applying Corollary~\ref{cor:mu-part} we obtain that
$v\in \Stab^\mu_\val(\alpha/\CC)$ if and only if $\alpha\in \CX+\mu_n^\FC \Leftrightarrow
v+\alpha\in \CX+\mu_n^\FC$, for every $\CX\subseteq \FC^n$
that is $\CL_\val$-definable over $\CC$.

If $\CX\subseteq \FC^n$
 is $\CL_\val$-definable over $\CC$  and $u\in \CC^n$ then
 the set $u+\CX$ is $\CL_\val$-definable over $\CC$ as well. 
Thus for $v\in \CC^n$ we have that $v\in \Stab^\mu_\val(\alpha/\CC)$  
 if and only if for every $\CL_\val$-definable over $\CC$ set
 $\CX\subseteq \FC^n$ and every  $u\in
\CC^n$ we have 
$\alpha\in u+\CX+\mu_n^\FC \Leftrightarrow
v+\alpha\in u+\CX+\mu^n_\FC$.

For a set $\CX \subseteq \FC^n$ that is
$\CL_\val$-definable over $\CC$  let
\[F_\CX=\{ u\in \CC^n \colon \alpha\in u+\CX+\mu^n_\FC \}.\] 
Let $v\in \CC^n$.  It follows from the above discussion that 
$v\in \Stab^\mu_\val(\alpha/\CC)$ if and only if $-v+F_\CX=F_\CX$,
equivalently $\CX_F=\CX_F+v$  
for every  
 $\CX \subseteq \FC^n$ that is
$\CL_\val$-definable over $\CC$.

For a set $\CX \subseteq \FC^n$ that is
$\CL_\val$-definable over $\CC$ let
\[G_\mathcal{X}=\{ v\in \CC^n \colon v+F_\mathcal{X} = F_\mathcal{X}\} \]
be the stabilizer of the set $F_\mathcal{X}$ in  $(\CC^n,+)$.

 Obviously each
$G_\mathcal{X}$ is a subgroup of $(\CC^n,+)$.
By Proposition~\ref{prop:ACVF} every  $F_\CX$ is a constructible subset of $\CC^n$. Hence each
$G_\CX$ is an algebraic  subgroup of $(\CC^n,+)$.

As we observe above 
\[
\Stab^\mu_\val(\alpha/\CC)=\cap\{ G_\mathcal{X} \colon \mathcal X
  \subseteq \FC^n \text{ is  $\CL_\val$-definable over $\CC$}\}.
\]

\medskip
Thus
$\Stab^\mu_\val(\alpha/\CC)$ is an intersection of
  algebraic subgroups of $(\CC^n;+)$.  Since the field $\CC$
 satisfies  the Decreasing Chan Condition   on algebraic subgroups,
$\Stab^\mu_\val(\alpha/\CC)$ is an intersection of finitely many
$G_\CX$, hence is algebraic.

\medskip

(2). Assume  $\alpha\in \FC^n$ is unbounded.

Let $p_\sa(x)=\tp_\sa(\alpha/\RR)$ be  the
semialgebraic type of  $\alpha$ over $\RR$.
Since by Fact~\ref{fact:mustab}(2) $\Stab^\mu_\sa(p)$ is infinite,
it is
sufficient to show that $\Stab^\mu_\sa(p)\subseteq
\Stab^\mu_\val(\alpha/\CC)$.

Let $v\in \Stab^\mu_\sa(p)$. We have $v+\alpha \in p_\sa(\FC)+\mu_\FC^n$. By
Theorem~\ref{thm:acvf-semi} (see also Remark~\ref{rem:val-sm}), $p_\sa(\FC)\subseteq
\tp_\val(\alpha/\CC)$. Thus $v+\alpha \in p_\sa(\FC)+\mu_\FR^{2n}
\subseteq \CP_\alpha^\mu$.  By Corollary~\ref{cor:mu-part}, 
$\CP_\alpha^\mu=\CP_{v+\alpha}^\mu$, and, by  Lemma~\ref{lem:mu-add}, 
$v\in \Stab_\val^\mu(\alpha/\CC)$.
\end{proof}

\section{Affine asymptotes}
\label{sec:affine-assymptotes}

Using  an idea of Ullmo and Yafaev from \cite[Section 2]{UY} we introduce  the notion
of affine asymptotes.

As usual if $V$ is a vector space over $\CC$, then a translate of a $\CC$-linear
subspace of $V$  is called \emph{an affine $\CC$-subspace of $V$} or \emph{a $\CC$-flat
subset of $V$}.

\begin{defn} Let  $\alpha\in \FC^n$. The smallest  $\CC$-flat subset
$A\subseteq \CC^n$  with $\alpha\in A^\sharp+\mu^n_\FC$ is called
\emph{the asymptotic $\CC$-flat of $\alpha$} or just \emph{the $\CC$-flat of
$\alpha$} and is denoted by $A_\alpha^\CC$.
\end{defn}

To justify the above definition we need to show that such smallest
$\CC$-flat  exists. It follows from the following proposition.

\begin{prop}\label{prop:as-inter} Let $\alpha\in \FC^n$.
If  $A_1,A_2\subseteq \CC^n$ are $\CC$-flat subsets with $\alpha\in
A_i^\sharp+\mu_\FC^n$, $i=1,2$,  then $\alpha\in {(A_1\cap
A_2)}^\sharp+\mu_\FC^n$.
\end{prop}
\begin{proof}
By an elementary linear algebra, $\CC$-flat subsets of $\CC^n$ are exactly solution
sets of the linear systems $Mx=r$, for an $m\times n$-matrix $M$ over $\CC$ and
$r\in \CC^m$ ($m$ is arbitrary).

We need a claim.
\begin{claim}\label{claim:lin-alg-cont}
Let $A$ be a $\CC$-flat subset of $\CC^n$
  given as the solution set of $Mx= r$, where $M$ is  an $m\times n$
  matrix over $\CC$.  Then $\alpha\in
  A^\sharp+\mu_\FC^n$ if and only if
  $M\alpha\in  r+\mu_\FC^m$.
\end{claim}
\begin{proof}[Proof of the claim.]
 If $\alpha\in A^\sharp+\mu_\FC^n$  then $\alpha\in \beta+\mu_\FC^n$ for some
 $\beta\in A^\sharp$, and
$M\alpha \in M\beta+M\mu_\FC^n  \subseteq r+\mu_\FC^m$.

For the right to left direction, assume $M\alpha\in  r +\mu_\FC^m$.
Replacing $\CC^m$ by the range of $M$ if needed we will assume
that the range of $M$ is the whole $\CC^m$.

Let $V_0\subseteq \CC^n$ be the kernel of $M$. We choose $V_1\subseteq \CC^n$ a
$\CC$-subspace complementary to $V_0$, so $\CC^n=V_0\oplus V_1$. We write $\alpha$
as $\alpha=\alpha_0+\alpha_1$ with $\alpha_0\in V_0^\sharp$, $\alpha_1\in
V_1^\sharp$. Since the restriction of $M$ to $V_1$ is an invertible $\CC$-linear map
from $V_1$ onto $\CC^m$  and $M\alpha_1\in r+\mu_\FC^m$, there is $\beta_1\in
V_1^\sharp$ with $\beta_1\in \alpha_1+\mu_\FC^n$ and $M\beta_1=r$. Obviously
$\alpha_0+\beta_1\in A^\sharp+\mu_\FC^n$, hence $\alpha\in A^\sharp+\mu_\FC^n$.

This finishes the proof of the claim.
 \end{proof}

We now proceed with the proof of the proposition.

Let $A=A_1\cap A_2$. For $i=1,2$ we choose   $m_i{\times}n$-matrices
  $M_i$ over $\CC$  and $ r_i \in \CC^{m_i}$,
such that $A_i$ is the solution set of $M_ix= r_i$.
Then $A$ is the solution set of $Mx= r$, where
$M$ is the $m\times n$ matrix $\begin{pmatrix} M_1 \\
  M_2  \end{pmatrix}$ and 
$ r =\begin{pmatrix} r_1 \\
  r_2  \end{pmatrix}$.

Using Claim~\ref{claim:lin-alg-cont} for $\alpha$ and  $A_1$ and $A_2$, we see that
$M\alpha=r+\epsilon$ for some $\epsilon \in \mu_\FC^{m_1+m_2}$. Using
Claim~\ref{claim:lin-alg-cont}
again we see that $\alpha\in A^\sharp +\mu_\FC^n$.
\end{proof}

\begin{defn}
 For a constructible set $X \subseteq \CC^n$ we will denote by $\CA^\CC(X)$ the set of
all  $\CC$-flats of  elements of $X^\sharp$, namely
\[ \CA^\CC(X)=\{ A^\CC_\alpha \colon \alpha \in X^\sharp \}.
 \]
\end{defn}

We say that  a family $\CF$ of subsets of $\CC^n$ is \emph{a
  constructible family } if there is a constructible set $T\subseteq
\CC^k$ and a constructible set $Y\subseteq \CC^n\times T$ such that
$\CF=\{ Y_t \colon t\in T\}$, where $Y_t$ is the fiber of $Y$ above $t$.

The next theorem follows easily from Proposition~\ref{prop:ACVF}.
\begin{thm}\label{thm:as-def}
 If $X\subseteq \CC^n$ is a constructible set then the family
    $\CA^\CC(X)$ is also constructible.
\end{thm}

\begin{sample}
(1). Consider the curve $X\subseteq  \CC^2$ given by $xy=1$.

 Let $\alpha=(\alpha_1,\alpha_2)\in X^\sharp$. If $\alpha$ is bounded,
 i.e. $\alpha\in \CO^2_\FC$,  then $A_\alpha^\CC$ is just the point
 $\st(\alpha)$.
Also notice that since $X$ is closed we have $\st(\alpha)\in X$.

 If $\alpha$ is unbounded then either $\alpha_1\notin \CO_\FC$ and
 $\alpha_2\in \mu_\FC$ or $\alpha_1 \in \mu_\FC $ and
 $\alpha_2 \notin \CO_\FC$.

In the first case we get $A_\alpha^\CC=\CC\times 0$, and in the
second $A_\alpha^\CC= 0 \times \CC$.

Thus $\CA^\CC(X)$ consists of all points in $X$ together with two lines
$\CC{\times}0$ and $0{\times}\CC$.

\medskip

(2). Consider the curve $X\subseteq  \CC^2$ given by $y=x^2$.

 Let $\alpha \in X^\sharp$. Again if $\alpha$ is bounded
 then $A_\alpha^\CC$ is
 $\st(\alpha)$.

If $\alpha$ is unbounded then $\alpha=(\xi,\xi^2)$ with $\xi\notin \CO_\FC$.  It is
easy to see that $\xi$ and $\xi^2$ are $\CC$-independent modulo 
$\CO_\FC$, i.e.\ for
$c_1,c_2\in \CC$, if  $c_1\xi+c_2\xi^2\in \CO_\FC$ then $c_1=c_2=0$.  It follows
then that $A_\alpha^\CC=\CC^2$.

Thus in this case $\CA^\CC(X)$ consists of all points in  $X$ together with the
plane $\CC^2$.

\medskip
(3). If we take $X=\CC^2$ then $\CA^\CC(X)$  will be the set of all
$\CC$-flat subsets of $\CC^2$.
\end{sample}

\begin{rem}\label{re:ana}
Since  the theory of algebraically closed fields of characteristic zero with an
embedded residue field is complete  we can define
$\mathcal{A}^\CC(X)$ using the field  of  convergent  Puiseux series instead of
$\FC$, and get an analytic interpretation of $\mathcal{A}^\CC(X)$ as follows.

We denote by  $\CC(\{z\})$  the field of germs meromorphic functions at $0\in
\CC$.

For  constructible set $X\subseteq \CC^n$  we denote by
$X_\mathrm{an}$ the set of  $\CC(\{z\})$-points on $X$,  in other words
\[ X_\mathrm{an}=\{(f_1,\dotsc,f_n)\in \CC(\{x\})^n \colon (f_1(z),\dotsc,f_n(z))
\in X \text{ for all $z$ near $0$} \}. \]

 For  $f\in X_\mathrm{an}$ let $A_f\subseteq \CC^n$ be
the smallest $\CC$-flat subset  of $\CC^n$ such that the distance from $f(z)$ to
$A_z$ tends to $0$ as $z$ approaches $0$.  Then $\mathcal{A}^\CC(X)=\{ A_f
\colon f\in X_\mathrm{an}\}$.
\end{rem}

\medskip
We  need some basic properties of $A_\alpha^\CC$.

\begin{lem}\label{lem:Aalpha-bounded}
Let $\alpha\in \FC^n$. Then   $\alpha$ is bounded (i.e. $\alpha\in \CO_\FC^n$) if
and only if $\dim_\CC(A^\CC_\alpha)=0$. Also if $\dim_\CC(A_\alpha^\CC)=0$ then
$A_\alpha^\CC=\st(\alpha)$.
  \end{lem}
  \begin{proof}
It follows from the definition of $A_\alpha^\CC$ that $\dim_\CC (A_\alpha^\CC)=0$ if
and only if $\alpha\in c+\mu_\FC^n$ for some $c\in \CC^n$, i.e. $\alpha$ is bounded
and $A_\alpha^\CC=c$.
   \end{proof}

\begin{lem}\label{lem:ha-is-in-la}
If $\alpha\in \FC^n$ then   $A_\alpha^\CC$ is invariant under
translations by elements of $\Stab^\mu_\val(\alpha/\CC)$.
\end{lem}
\begin{proof}
Notice that if $\alpha$ and $\beta$ realize the same $\CL_\val$-type
then $A^\CC_\alpha=A^\CC_\beta$. Moreover, the same is true if
$\CP_\alpha^\mu=\CP_\beta^\mu$, where $\CP_\alpha^\mu$ as in
Definition~\ref{defn:val-stab}

Now, if $v\in \Stab^\mu_\val(\alpha/\CC)$ then
$\CP^\mu_\alpha=v+\CP^\mu_\alpha=\CP_{v+\alpha}^\mu$. By what we just noted,
\[A^\CC_{\alpha}=A^\CC_{v+\alpha}=v+A^{\CC}_{\alpha},\]
as claimed.
\end{proof}

The lemma below follows easily from Lemma~\ref{lem:ha-is-in-la} and we leave its
proof to the reader.
 \begin{lem}\label{lem:la-subspace}
Let $\alpha\in \FC^n$, $H_\alpha=\Stab^\mu_\val(\alpha/\CC)$ and $V_1\subseteq
\CC^n$ a subspace complementary to $H_\alpha$. Let $\pi\colon \CC^n\to V_1$ be the
projection along $H_\alpha$ and $\alpha_1=\pi(\alpha)$. Then
$A^\CC_\alpha=H_\alpha\oplus A^\CC_{\alpha_1}$.
\end{lem}

\section{Describing $\cl(X+\Lambda)$ using asymptotic flats}
\label{sec:first-reduction}

The main goal of this section is to   describe $\cl(X+\Lambda)$ as the 
union: 
\[ \cl(X+\Lambda)=\bigcup_{A\in \CA^\CC(X)} \cl(A+\Lambda). \]

The next proposition is the key ingredient. 
\begin{prop}\label{prop:main} Let $\alpha\in \FC^n$ and let $\Lambda \leq
  \CC^n$ be an additive  subgroup  whose
$\RR$-span is $\CC^n$.
 Let $p(x)=\tp_\val(\alpha/\CC)$.
Then $\st(p(\FC)+\Lambda^\sharp)$ is a closed subset of $\CC^n$ and
\[ \cl(A_\alpha^\CC +\Lambda) = \st(p(\FC)+\Lambda^\sharp).\]
 \end{prop}
\begin{proof}
Let $r_\alpha$ be a semialgebraic type as in
Corollary~\ref{cor:acvf-semi}, namely $\rho_\alpha(\FC)=p(\FC)$.
By
Claim~\ref{claim:cl-stand-part},  $\st(r_\alpha(\FC)+\Lambda^\sharp)$ is
closed.  Hence $ \st(p(\FC)+\Lambda^\sharp)$ is closed as well.

\medskip
For the inclusion $\cl(A_\alpha^\CC +\Lambda) \supseteq \st(p(\FC)+\Lambda^\sharp)
$, let
 $\beta\in p(\FC)$. We need to show that
$\st(\beta+\Lambda^\sharp)\subseteq \cl(A_\alpha^\CC + \Lambda)$.
Notice that since $\beta$ and $\alpha$ have the same $\CL_\val$-type
over $\CC$ we have $A_\beta^\CC=A_\alpha^\CC$.

  By the definition of  $A_\beta^\CC$, there is
$\gamma\in {(A_\beta^\CC)}^\sharp$ with $\beta\in
\gamma+\mu_\FC^n$. Hence  we have
$\st(\beta+\Lambda^\sharp)\subseteq \st( {(A_\beta^\CC)}^\sharp +\Lambda^\sharp)$,
and by Claim~\ref{sec:clos-stand-part},
\[
\st(\beta+\Lambda^\sharp) \subseteq \cl(A_\beta^\CC +\Lambda)=\cl(A_\alpha^\CC +\Lambda).\]

\medskip

We are left to show  the inclusion
\[ \cl(A_\alpha^\CC +\Lambda) \subseteq \st(p(\FC)+\Lambda^\sharp). \]
Since the right side is invariant under translations by elements of $\Lambda$ and is
closed it is sufficient to prove
\[ A_\alpha^\CC \subseteq \st(p(\FC)+\Lambda^\sharp). \]

 We proceed  by induction on $\dim_\CC(A^\CC_\alpha)$, and
to simplify notation we denote  $\CC^n$ by $V$.

\medskip
\noindent\textbf{Base case: $\dim_\CC(A_\alpha^\CC)=0$.}  Then, by
Lemma~\ref{lem:Aalpha-bounded},
$\alpha$ is bounded with $A_\alpha^\CC=\st(\alpha)$,
and the proposition is
 trivial in this case.

\medskip
\noindent\textbf{Inductive step.}  Assume $\dim_\CC(A_\alpha^\CC)>0$,   hence
$\alpha$ is unbounded.

Let $H_\alpha=\Stab^\mu_\val(\alpha/\CC)$. By Theorem~\ref{thm:mu-def}(2), $\dim_\CC(H_\alpha)>0$.

Choose  a $\CC$-subspace  $V_1\subset V$,
complementary to $H_\alpha$,  i.e
\[ V=H_\alpha\oplus V_1, \]
and let  $\pi\colon V\to V_1$ be the projection of $V$ onto $V_1$ along
$H_\alpha$.

 We can
write $\alpha$ as
\[ \alpha=\alpha_0+\alpha_1 \text{ with $\alpha_0\in H_\alpha^\sharp$ and
$\alpha_1=\pi(\alpha)\in V_1^\sharp$ }.\]

By Lemma~\ref{lem:la-subspace} we have
 $A_\alpha^\CC=H_\alpha\oplus A_{\alpha_1}^\CC$, so
$\dim_\CC(A_{\alpha_1}^\CC)<\dim_\CC(A_\alpha^\CC)$.

To prove the proposition, it is sufficient to show that for any
$a\in A_{\alpha_1}^\CC$ we have
\begin{equation}
  \label{eq:2}
a+H_\alpha \subseteq \st(p(\FC)+\Lambda^\sharp).
\end{equation}

We fix
$a\in A_{\alpha_1}^\CC$.

We now apply the inductive hypothesis to $A_{\alpha_1}^\CC$ and
$\Lambda_1=H_\alpha+\Lambda$ (which clearly still $\RR$-spans $V$). We obtain
\[ a\in \st(p_1(\FC)+H_\alpha^\sharp+\Lambda^\sharp),\]
where $p_1=\tp_\val(\alpha_1/\CC)$.

Thus we have
\[
a =\st(\beta_1 +h+\lambda) \text{ for some } \beta_1\in p_1(\FC),
h\in H_\alpha^\sharp,  \lambda \in \Lambda^\sharp.
\]
By Corollary~\ref{cor:acvf-sat}, $\pi$ maps $p(\FC)$ onto $p_1(\FC)$,
hence
there is $\beta\in p(\FC)$ with
$\pi(\beta)=\beta_1$. Thus $\beta_1=\beta+h'$ for some
$h'\in H_\alpha^\sharp$ and
\[ a =\st(\beta +h''+\lambda) \text{ with } \beta \in p(\FC),
h''\in H_\alpha^\sharp,  \lambda \in \Lambda^\sharp.
\]

Let $H=H_\alpha^\Lambda$. Since $H$ has an  $\RR$-basis in $\Lambda$, there
 is a compact subset $F\subseteq H$ with $H\subseteq F+\Lambda$.

We now use the fact that $\FR_\full$ is an elementary extension of $\RR_\full$.
Since $h'' \in H^\sharp$, there is $\lambda_1\in \Lambda^\sharp$ with
$\lambda_1-h''\in F^\sharp$. Because $F$ is compact, there is $h^*\in F^\sharp\sub
H^\sharp$ with $h^*=\st(\lambda_1-h'')$. Thus
\[ a+h^*=st(\beta+\lambda_1+\lambda)\in  \st(\beta+\Lambda^\sharp),\]
and
\[ a+h^*+H_\alpha \subseteq \st(\beta+H_\alpha+\Lambda^\sharp).\]
By the definition of $\mu$-stabilizers, $\beta+H_\alpha \subseteq
p(\FC)+\mu_\FC^n$, hence
 \[ a+h^*+H_\alpha \subseteq \st(p(\FC)+\Lambda^\sharp).\]

Since  the right side is closed and
invariant under translations by elements of $\Lambda$, we conclude that
\[ a+h^*+\cl(H_\alpha+\Lambda) \subseteq \st(p(\FC)+\Lambda^\sharp).\]

By Fact~\ref{fact:kron},  $H\subseteq \cl(H_\alpha+\Lambda)$, and
since $h^*\in H$,  we have
\[ a +H \subseteq \st(p(\FC)+\Lambda^\sharp),\]
hence
\[ a +H_\alpha\subseteq \st(p(\FC)+\Lambda^\sharp).\]
This proves \eqref{eq:2} and therefore the proposition.

\end{proof}

We can now deduce the main theorem of this section.

\begin{thm}\label{thm:main-1}
Let $X\subseteq \CC^n$ be a constructible set and $\Lambda\leq  \CC^n$
be an additive
subgroup whose $\RR$-span is the whole $\CC^n$.  Then
\[ \cl(X+\Lambda)=\bigcup_{A\in \CA^\CC(X)} \cl(A  + \Lambda),
\]
and the family of $\CC$-flats $\CA^\CC(X)$ is constructible.
\end{thm}
\begin{proof}
By Claim~\ref{claim:cl-stand-part} we have
\[
  \cl(X+\Lambda)=\st(X^\sharp + \Lambda^\sharp) =
\bigcup_{\alpha\in X^\sharp} \st(\alpha+\Lambda^\sharp) =\bigcup_{ p(x)\in
S^\CC_\val(X)}\st(p(\FC)+\Lambda^\sharp),
\]
where $ S^\CC_\val(X)=\{ \tp_\val(\alpha/\CC) \colon \alpha\in
X^\sharp\}$.

By Proposition~\ref{prop:main}, the union on the right equals
\[\bigcup_{A\in \CA^\CC(X)} \cl(A+\Lambda),\]
thus proving the required  equality.

By Theorem~\ref{thm:as-def}, the family $\CA^\CC(X)$ is constructible.
\end{proof}

\section{Completing the proof in the algebraic case}
\label{sec:second-reduction}

Our next goal is to show that in Theorem~\ref{thm:main-1} one can replace
the union $\bigcup_{A\in \CA^\CC(X)} \cl(A+\Lambda)$ by a finite union of sets of the
form $C_i+V_i^\Lambda+\Lambda$, where each $C_i$ is constructible and each $V_i$ 
a $\CC$-linear subspace of $\CC^n$.

\subsection{On families of affine subspaces}
\label{sec:famil-affine-subsp}

By $\Graff_k(\CC^n)$ we denote the Grassmannian
variety of all affine $k$-dimensional $\CC$-subspaces of $\CC^n$.
Each $\Graff_k(\CC^n)$ is a quasi-projective subvariety of some
 $\PP^{\hspace{0.05 em}l_k}(\CC)$, and  we often identify $A\in
\Graff_k(\CC^n)$ with the corresponding $\CC$-flat $A\subseteq
\CC^n$.

The Euclidean  topology on  $\Graff_k(\CC^n)$ induced by $\PP^{\hspace{0.05
em}l_k}(\CC)$ coincides with the topology induced by the distance function on the
set of all $\CC$-flats in $\CC^n$ that is defined as follows: Let  $A_1,A_2
\subseteq \CC^n$ be two $\CC$-flat subsets.  Let $\xi_1\in A_1$ be the point of
$A_1$ closest to the origin with respect to the Euclidean norm on $\CC^n$, and
similarly we choose $\xi_2\in A_2$.  Let $S_1=\{ v\in A_1 \colon \| v-\xi_1\|=1\}$
and $S_2= \{ v\in A_2 \colon \| v-\xi_2\|=1\}$.  The distance between $A_1$ and
$A_2$ is now defined to be the Hausdorff distance between $S_1$ and $S_2$.

For a $\CC$-flat $A\subseteq \CC^n$ we denote by $L(A)$ the
linear part of $A$, i.e.\ the linear subspace of $\CC^n$ such that $A$
is a translate of $L$.

\medskip

For a subset $T\subseteq \Graff_k(\CC^n)$ we denote by $\CCL(T)$ the $\CC$-linear span of
$\bigcup_{A\in T} L(A)$ in $\CC^n$. Slightly abusing notation, for   a
$\CC$-subspace $W\subseteq \CC^n$ of arbitrary  dimension we  let
\[ L^{-1}[W]=\{ A\in \Graff_k(\CC^n) \colon L(A) \text{ is a subspace of }
W\}.\]  It is not hard to see that $L^{-1}[W]$ is a Zariski closed subset of
$\Graff_k(\CC^n)$.

For  a $\CC$-flat $A\subseteq \CC^n$ and an additive subgroup $\Lambda\leq \CC^n$
whose $\RR$-span is the whole $\CC^n$ we denote by $A^\Lambda$ the $\CC$-flat
$a+{L(A)}^\Lambda$, where $a\in A$. Obviously the definition of $A^\Lambda$ does not
depend on the choice of $a$. Notice that if $\Lambda$ is
a lattice in $\CC^n$ then  $a+{L(A)}^\Lambda$ is the connected
component of $\cl(A+\Lambda)$ containing $a$.  

\begin{rem}
Let $A\subseteq \CC^n$ be a $\CC$-flat and $V\subseteq \CC^n$ a $\CC$-subspace with
$L(A)\subseteq V$. Then $A+V$ is also a $\CC$-flat with $L(A+V)=V$.  Also if
$A=a+L(A)$ then $A+V=a+V$.
\end{rem}

Recall that a constructible subset $T\subseteq \Graff_k(\CC^n)$ is
called  \emph{irreducible} if its Zariski closure is an irreducible
subvariety of  $\Graff_k(\CC^n)$

\begin{prop}\label{prop:neat-dense}
Let $T\subseteq \Graff_k(\CC^n)$ be an irreducible constructible set and
$V=\CCL(T)$.
 Let $\Lambda< \CC^n$ be a countable additive subgroup whose
 $\RR$-span is the whole $\CC^n$.
  \begin{enumerate}
  \item The set
  $\{A\in T \colon  {L(A)}^\Lambda =V^\Lambda\}$
  is topologically dense in $T$ with respect to the Euclidean topology on $\Graff_k(\CC^n)$.
\item  The set
$\bigcup_{A\in T} (A+\Lambda)$ is topologically dense in $\bigcup_{A\in
  T} (A+V+\Lambda)$ with respect to the Euclidean topology on
$\CC^n$.
\end{enumerate}
\end{prop}
\begin{proof}
(1).  Since $\Lambda$ is countable, there are at most  countably many
$\RR$-subspaces of $\CC^n$ defined over $\Lambda$, hence  by
the  Baire category  theorem
  it is sufficient to show that for any proper $\RR$-subspace $W \lneq
  V^\Lambda$ defined over $\Lambda$ the set
$T_W=\{ A\in T \colon L(A) \leq W\}$  is nowhere
  dense in $T$.
Let $W\lneq V^\Lambda$ be a proper $\RR$-subspace of $V^\Lambda$
defined over $\Lambda$.   Let $W'=W\cap iW$ be the largest
$\CC$-subspace $\CC^n$ contained in   $W$.   For $A\in T$  the space  $L(A)$ is a
$\CC$-subspace of $\CC^n$, hence $L(A) \leq W$ if and only if
$L(A)\leq W'$.  Thus $T_W=L^{-1}[W']$.

Since $T$ is
  irreducible and $L^{-1}[W']$ is Zariski closed in $\Graff_k(\CC^n)$,
the set $T_W$ is nowhere dense in $T$
  if and only if  $T \not\subseteq L^{-1}[W']$.  Assume
  $T \subseteq L^{-1}[W']$.  Then $V=\CC
  L(T) \subseteq W'\subseteq W$, and since $W$ is defined over $\Lambda$ we would
  have $V^\Lambda \subseteq W$, contradicting the assumption on $W$.

\medskip
(2).  For $A\in T$, let $\xi_A\in A$ be  the point on $A$ closest
to the origin with respect to the Euclidean metric on
$\CC^n$.
It is not hard to see that the map
$A\mapsto \xi_A$, as a map from $T$ to $\CC^n$,
is continuous with respect to  Euclidean topologies.

By Fact~\ref{fact:kron}, for any $\CC$-subspace $W\subseteq \CC^n$ and $\xi\in \CC^n$
we have $\xi+W^\Lambda +\Lambda \subseteq
cl(\xi+W+\Lambda)$. Therefore, writing each $A\in T$ as
$A=\xi_A+L(A)$, we have
\[ \bigcup_{A\in T}(\xi_A+{L(A)}^\Lambda+\Lambda)\subseteq \cl\left( \bigcup_{A\in T} (A+\Lambda)\right).\]
Thus it is sufficient to show that the set $\bigcup_{A\in
  T}(\xi_A+{L(A)}^\Lambda+\Lambda)$ is dense in $\bigcup_{A\in
  T}(\xi_A+V^\Lambda+\Lambda)$.

The latter follows from clause (1) and the continuity of the map
$A\mapsto \xi_A$.
 \end{proof}

\subsection{Proof of the main theorem in the algebraic case}
\label{sec:proof-main-theorem}

We can now prove  our main result in the algebraic case.

For a $\CC$-subspace $W\subseteq \CC^n$, we denote by $W^\perp$ its orthogonal
complement with respect to the standard inner product on $\CC^n$.

\begin{thm}\label{thm:main0} Let $X\subseteq \CC^n$ be  an algebraic subvariety.
There are $\CC$-subspaces  $V_1,\dotsc,V_m \subseteq
 \CC^n$ of positive dimension  and algebraic subvarieties
$C_1\subseteq V_1^\perp, \dotsc,C_m \subseteq V_m^\perp$
such that for any lattice
 $\Lambda<\CC^n$ we have
 \begin{equation}
   \label{eq:1}
    \cl(X+\Lambda) =(x+\Lambda)
    \cup\bigcup_{i=1}^m (C_i+V_i^\Lambda+\Lambda).
 \end{equation}

In addition,
\begin{enumerate}[(i)]
\item For each $i=1,\dotsc, m$, we have $\dim C_i<\dim X$. \item For each $V_i$ that
is maximal among $V_1,\dotsc, V_m$, the set $C_i\sub \RR^n$ is finite.
\end{enumerate}
\end{thm}
\begin{proof}
Notice that it is sufficient to find $C_i$'s  as above that are
constructible, and then
replace  each $C_i$ with its topological closure  if needed.

By Theorem~\ref{thm:main-1} we have
  \[ \cl(X+\Lambda)=\bigcup_{A\in \CA^\CC(X)} \cl(A+\Lambda),\]
and
we view $\CA^\CC(X)$ as a constructible subset of
$\Graff_0(\CC^n)\cup\dotsb \cup\Graff_n(\CC^n)$.

Since $X$ is a closed set, we have $\st(X^\sharp)=X$, and using   Lemma~\ref{lem:Aalpha-bounded}
we identify $\CA^\CC(X)\cap \Graff_0(\CC^n)$  with $X$.

Since every constructible subset  of $\Graff_k(\CC^n)$ is a finite union of
irreducible constructible sets, we can write $\CA^\CC(X)$ as
\begin{equation} \label{eq:Ti} \CA^\CC(X)=X\cup T_1\cup\dotsb \cup T_m,
\end{equation}
 where each $T_i$ is an irreducible  constructible subset  of
some $\Graff_{k_i}(\CC^n)$ with $k_i>0$.

Thus we have
\[ \cl(X+\Lambda)=(X+\Lambda)\cup\left(\bigcup_{i=1}^m
\bigcup_{A\in T_i} \cl(A+\Lambda)\right). \]

For each $i=1,\dotsc,m$, let  $V_i=\CCL(T_i)$. Obviously each $V_i$
has positive dimension.
 By Proposition~\ref{prop:neat-dense}(2),
every
closed set containing $\bigcup_{A\in T_i}(A+\Lambda)$ must also
contain $\bigcup_{A\in T_i}(A+V_i+\Lambda)$, hence
\[ \cl(X+\Lambda)=(X+\Lambda)\cup\left(\bigcup_{i=1}^m
\bigcup_{A\in T_i} \cl(A+V_i+\Lambda)\right). \]

Fix $i=\{1,\dotsc,m\}$. Since $L(A)\subseteq V_i$ for $A\in T_i$, we have that for
each $A\in T_i$  the intersection $(A+V_i) \cap V_i^\perp$ is a singleton that we
denote by $c_A$.  Obiously $A+V_i= c_A+V_i$. Let $C_i=\{ c_A \colon A\in T_i\}$. It is not
hard to see that $C_i$ is a constructible subset of $V_i^\perp$.

We have
\begin{multline*}
\cl(X+\Lambda)= (X+\Lambda) \cup\left(\bigcup_{i=1}^m
\bigcup_{A\in T_i} \cl(c_A+V_i+\Lambda)\right) =\\
(X+\Lambda)
\cup\left(\bigcup_{i=1}^m
\bigcup_{c\in C_i} (c+V_i^\Lambda+\Lambda)\right)= \\
=(X+\Lambda)\cup\bigcup_{i=1}^m(C_i+V_i^\Lambda+\Lambda).
\end{multline*}

This  finishes the proof of the main part of Theorem~\ref{thm:main0}.

\medskip

\noindent \emph{Proof of Clause (i).}  Although, considering $X$ as a semialgebraic
set, Clause (i) can be derived from the corresponding clause in the o-minimal case,
we also provide an algebraic argument.

 Let $V$ be one of $V_i's$, $i=1,\dotsc,m,$ and $C\subseteq
V^\perp$ the corresponding constructible subset.  We need to show that
$\dim_\CC(C)<\dim_\CC(X)$.

By our choice of $V$ and $C$,  for each $c\in C$ there is $\alpha
\in X^\sharp \setminus \CO_\FC^n$ with
$A_\alpha^\CC\subseteq c+V$, equivalently  $\alpha\in c+V^\sharp
+\mu_\FC^n$.

Changing  coordinates, we may assume that $V=\CC^k$, $V^\perp=\CC^l$
with $k+l=n$ and we write $\CC^n$ as $\CC^l\times \CC^k$.
 Let $\overline{X}^Z$ be the Zariski closure of $X$ in $\CC^l\times
 \PP^k(\CC)$
under the embedding  $\CC^l\times \CC^k \hookrightarrow
\CC^l\times \PP^k(\CC)$.

Since $\dim_\CC(\overline{X}^Z\setminus X) < \dim_\CC(X)$,  it is sufficient
to show that  $C$ is contained in the projection of $\overline{X}^Z\setminus X$ into $\CC^l$.

Let $\bar c=(c_1,\dotsc,c_l)\in C$.
Choose $\bar \alpha\in X^\sharp \setminus \CO^n_\FC$ with
$\bar \alpha\in \bar c+{(\CC^k)}^\sharp +\mu^n_\FC$.

We can write   $\bar \alpha$ as
$\bar \alpha=(\bar c', \bar \alpha')$, with
$\bar c'= \bar c +\mu_\FC^l$   and
$\bar \alpha' \in \FC^k+\mu_\FC^k$.
Notice that we must have  $\bar \alpha'\notin \CO^k_\FC$.

Choose $\varepsilon\in \FC$ so that $\varepsilon\bar \alpha' \in
\CO_\FC^k \setminus \mu_\FC^k$, (for example we can take
$\varepsilon=\frac{1}{\alpha_i'}$ where $\alpha_i' $ is a component of
$\bar\alpha'$ with smallest valuation).
Since $\bar \alpha'\not \in \CO_\FC^k$ we must have $\varepsilon\in
\mu_\FC$. Let $\bar a= (a_1,\dotsc,a_k)=\st(\varepsilon\bar \alpha')$.

We use $\bar x=(x_1,\dotsc,x_l)$ for coordinates in $\CC^l$ and
$[y_0:y_1:\dotsc:y_k]$  for  homogeneous coordinates in $\PP^k(\CC)$,
hence $\CC^k$ is identified with points represented by
homogeneous coordinates  $[1:y_1:\dotsc:y_k]$.

We claim that  $(c_1,\dotsc, c_l; 0:a_1:\dotsc : a_k)\in \overline{X}^Z$, and
since it is not in  $\CC^l\times \CC^k$ we would be done.

To show  that $(c_1,\dotsc, c_l; 0:a_1:\dotsc : a_k)\in
\overline{X}^Z$ we need to check that
$p(\bar c,0,\bar a)=0$
for every polynomial $p(x_1,\dotsc,x_l,y_0,\dotsc,y_k)\in\CC[\bar x,
\bar y]$ that is homogeneous in $\bar y$   and with
$ p(x_1,\dotsc,x_l, 1, y_1\dotsc,y_k)$  vanishing on $X$.

Let $p(x_1,\dotsc,x_l,y_0,\dotsc,y_k)$ be such  polynomial.
Since $\bar\alpha\in X^\sharp$ we have $p(\bar c',1,\bar \alpha')=0$.  Since $p$
is homogeneous in $\bar y$, we have $p(\bar c', \varepsilon,
\varepsilon\bar \alpha')=\varepsilon^s p(\bar c', 1, \bar \alpha')$ for
some $s\in \NN$, hence $p(\bar c', \varepsilon, \varepsilon\bar \alpha')=0$.

Since $p$ is a polynomial over $\CC$, $\bar c=\st(\bar c')$, $\varepsilon\in
\mu_\FC$, and $\bar a=\st( \varepsilon\bar \alpha')$, we have  $p(\bar c, 0, \bar
a)=0$, i.e.\ what we need. That finishes the proof of clause (i).

\end{proof}

\noindent \emph{Proof of Clause (ii).}  Let $T\subseteq \Graff_{k_i}(\CC)$ be one of
the $T_i$'s in (\ref{eq:Ti}), and let  $V$ and $C$ be the corresponding $V_i $ and
$C_i$. Thus $T\sub \CA^\CC(X)$ is an irreducible constructible set,
$V=\CCL(T)$ and
\[ C=\{ a\in V^\perp \colon
A\subseteq a+V \text{ for some } A\in T\}.\]

 Assume $C$ is
infinite.  Since $C$ is constructible it is unbounded with respect to the Euclidean
metric on $\CC^n$.  We will show that there is   $\alpha^*\in X^\sharp$  with
$L(A_{\alpha^*}^\CC)$ properly containing $V$. This would imply that
$A_{\alpha*}^\CC\in T_j$ for some $j=1,\dotsc,m$, and $L(A_{\alpha^*}^\CC)\subseteq
V_j$. Therefore such $V$ can not be maximal among the $V_i's$

\medskip

In order to find such an $\alpha^*$ we  use the following observation which can be
deduced from the definition of asymptotic  $\CC$-flats: for a $\CC$-subspace
$W\subseteq  \CC^n$ and $\alpha\in \FC^n$ we have
\[ L(A_\alpha^\CC)\subseteq  W \text{  if and only if  }\alpha\in
W^\sharp+\CO_\FC^n. \]

\medskip

Let $\Sigma$ be the collection of all  $\CC$-subspaces  $W \subseteq \CC^n$ with
$V\not\subseteq W$.

By the above observation,  to prove the proposition, we need  to show  that there is
$\alpha^*\in X^\sharp$ with $\alpha^*\not \in (V^\sharp+\CO_\FC^n)$ (hence
$L(A^\CC_\alpha)\not\subseteq V$), and also $\alpha^*\not \in (W^\sharp+\CO_\FC^n)$, for
any $W\in \Sigma$ (hence $V\subseteq L(A^\CC_\alpha)$).

Assume no such $\alpha^*$ as above exists. Then the intersection
\[
X^\sharp\cap \{ u\in  \FC^n \colon u\notin (V^\sharp+\CO_\FC^n) \} \cap\bigcap_{W\in
\Sigma} \{ u\in  \FC^n  \colon u \notin (W^\sharp+\CO_\FC^n) \}
\]
would be  empty.

Notice that every set in the above intersection is  pro-semialgebraic over $\RR$.

By the saturation of $\FR$, it follows then that there are
$W_1,\dotsc,W_l\in\Sigma$ and $R\in \RR$  such that
\[ X^\sharp \subseteq
( V^\sharp +{B^c(0,R)}^\sharp) \cup\bigcup_{i=1}^l ( W_i^\sharp +{B^c(0,R)}^\sharp),
 \]
where $B^c(0,R)$ is a closed ball in $\CC^n$ of radius $R$ centered at the origin.

It follows then that for every $A\in \CA^\CC(X)$
\begin{enumerate}
\item either $A\subseteq V+B^c(0,R)$; \item  or  $A \subseteq W_i+B^c(0,R)$ for some
$i=1,\dotsc,l$.
\end{enumerate}
Notice that in the second case we have $L(A)\subseteq W_i$.

For $j=1,\dotsc, l,$ let $Z_j=T \cap L^{-1}[W_j]$. Since $\CCL (T)=V$,
$V\not\subseteq  W_j$, and $T$ is irreducible we have $\dim_\CC Z_j<\dim_\CC T$.

The function  $h\colon T\to C$ that assigns to each $A$ the unique $c_A\in C$ with
$A\subseteq c_A+V$ is continuous with respect to the Euclidean topologies and
surjective. Since $C$ is unbounded,
 the set $h^{-1}(C\smallsetminus B^c(0,R))$ is 
a non-empty  open subset in the Euclidean topology of  $T$. By (1) and
(2)  it is covered by subsets  $Z_1,\dotsc,Z_l$ of  smaller
dimension. A contradiction. 

That finishes the proof of  Clause
(ii) and with it the proof of the main theorem in the algebraic
case.

\section{The o-minimal result}
\label{sec:o-minimal-result}

 In this section we prove the Main Theorem in the o-minimal case.
 The proof follows closely that of Theorem~\ref{thm:main0} so we shall be brief.

\textbf{In this section ``definable''
  means $\CL_\om$-definable in
structures $\RR_{\om}$ or $\FR_{\om}$. We  use $\dim_\RR$ to
denote the o-minimal dimension of definable sets}.

\subsection{Asymptotic $\RR$-flats}

Similarly to the algebraic case, for $\alpha\in \FR^n$ we let $A^\RR_{\alpha}$ be the
smallest affine $\RR$-subspace $A\sub \RR^n$ such that $\alpha\in A+\mu^n_{\FR}$.
The proof of its existence is identical to the proof in the complex case. We call
$A^\RR_{\alpha}$ \emph{the asymptotic $\RR$-flat of $\alpha$} or just \emph{the
$\RR$-flat of $\alpha$}. For $X\sub \RR^n$  definable in $\RR_{\om}$, we
define
\[\mathcal A^\RR(X)=\{A^\RR_{\alpha}:\alpha\in X^\#\}.\]

We can now prove the analogue of Theorem~\ref{thm:as-def}, using the theory of tame
pairs.
\begin{thm}\label{thm:asR-def} Let $X\sub \RR^n$ be a set definable  in $\RR_{\om}$. Then the family
of affine space $\mathcal A^\RR(X)$ is also definable in $\RR_{\om}$.
\end{thm}
\begin{proof} We consider the structure obtained by expanding the
o-minimal structure $\FR_{\om}$ by a predicate symbol  for the
real field and a function symbol $\st(x)$ for the standard part
map from $\CO_{\FR}$ into $\RR$. Thus we are working with a
pair of o-minimal structures $(\FR_{\om}, \RR_{\om}, \st)$ in
which the first structure is an elementary extension of the second
one and in addition the latter is Dedekind complete in the first.
Such an extension is called {\em tame }
 and if we let
$T$ be the theory of $\FR_{\om}$ then the theory of $(\FR_{\om}, \RR_{\om},
\st)$ is denoted  $T_{\tame}$. The model theory  of tame extensions was studied by
van den Dries and others, e.g.\ see \cite[Section 8]{lou-limit}. The main result
we need states:
\begin{fact}[{\cite[Proposition 8.1]{lou-limit}}] If $X\sub \RR^n$ is definable in $(\FR_{\om}, \RR_{\om}, \st)$
then $X$ is definable in $\RR_{\om}$. Moreover, if $\CF$ is a family of subsets of
$\RR^n$ that is definable in $(\FR_{\om}, \RR_{\om}, \st)$, then it is also
definable in $\RR_{\om}$.
\end{fact}
To be precise, it is only the first part of the above result which is proved in
\cite{lou-limit}, but the second part follows immediately by working in an arbitrary
model of $T_\tame$.

To complete the proof of Theorem~\ref{thm:asR-def}, we just need to observe that the
family $\mathcal A^\RR(X)$ is definable in $(\FR_{\om}, \RR_{\om}, \st)$.\end{proof}

 The
next step towards proving the Main Theorem  is the following analogue of Proposition~\ref{prop:main}.

\begin{prop}\label{omin-prop} For $\alpha\in \FR^n$  and
 an additive  subgroup $\Lambda \leq
  \RR^n$  whose
$\RR$-span is $\RR^n$,
 let $p(x)=\tp_{\om}(\alpha/\RR)$.
Then $\st(p(\FR)+\Lambda^\sharp)$ is a closed subset of $\RR^n$ and
\[ \cl(A_\alpha^\RR +\Lambda) = \st(p(\FR)+\Lambda^\sharp).\]
 \end{prop}
\begin{proof} The  inclusion $\supseteq$ is similar to the argument in Proposition~\ref{prop:main}: We need to see that every element of the form
$a=\st(\beta+\lambda)$ for $\beta\in p(\FR)$ and $\lambda\in \Lambda^\sharp$
belongs to $\cl(A^{\RR}_{\alpha}+\Lambda)$. It is easy to see that
$A^{\RR}_{\beta}=A^{\RR}_{\alpha}$ and hence there exists $\beta'\in
{(A^{\RR}_{\alpha})}^\sharp$ such that $\beta-\beta'\in \mu_{\FR}^n$. It follows
that
\[a=\st(\beta'+\lambda)\in
\st({(A^{\RR}_{\alpha})}^\sharp+\Lambda^\sharp)=\cl(A^{\RR}_\alpha+\Lambda).\]

 We need to prove the  inclusion  $\cl(A_\alpha^\RR +\Lambda)
 \subseteq \st(p(\FR)+\Lambda^\sharp)$. As in the algebraic case,
we use induction on
$\dim A^\RR_{\alpha}$
 and  assume that $\dim_{\RR} A^\RR_{\alpha}>0$,  so
$\alpha\notin \CO_{\FR}^n$.

We denote by  $R_{\alpha}$  the $\mu$-stabilizer of the o-minimal type $p$, namely
$R_\alpha =\Stab_{\om}^\mu(\alpha/\RR)$, as introduced in Section~\ref{sec:o-minimal-case}.
 By
Fact~\ref{fact:mustab}, $\dim_{\RR} R_\alpha>0$.  The group
$R_{\alpha}$ is an $\CL_{\om}$-definable subgroup of $(\RR^n,+)$
 and hence it is an $\RR$-subspace. We write
$\RR^n=R_\alpha\oplus V_1$ for some complementary $\RR$-space
$V_1$ (with $\dim_{\RR} V_1<n$) and let $\pi:\RR^n\to V_1$ be the
projection along $R_{\alpha}$. We write accordingly
$\alpha=\alpha_0+\alpha_1$.

We let $p_1=\tp_{\om}(\alpha_1/\RR)$. Using the saturation of $\FR$ we have
\[\pi(p(\FR))=p_1(\FR).\]
(Note that the above equality is immediate, unlike the algebraic case
where we had to use Corollary~\ref{cor:acvf-proj}.)

 As in Lemma~\ref{lem:ha-is-in-la}, we have $R_\alpha\sub L(A^\RR_\alpha)$. It
follows, as in Lemma~\ref{lem:la-subspace}, that
$A^\RR_{\alpha}=R_{\alpha}+A^\RR_{\alpha_1}$. Thus, it is enough to show that for
any $a\in A^\RR_{\alpha_1}$, we have $a+R_{\alpha}\sub \st(p(\FR)+\Lambda^{\#})$.

The remaining argument is identical to the proof of Proposition~\ref{prop:main} so
we omit it.
 \end{proof}

We can now conclude the o-minimal analogue of Theorem~\ref{thm:main-1}.
\begin{thm}\label{cor:omin} If  $X\sub \RR^n$ is a definable set
  then
\[\cl(X+\Lambda)=\bigcup_{A\in \CA^\RR(X)} \cl(A+\Lambda).\]
\end{thm}
\begin{proof} By Claim~\ref{claim:cl-stand-part} (1),
\[\cl(X+\Lambda)=\bigcup_{\alpha\in
X^\sharp}\st(\alpha+\Lambda^\sharp)=\bigcup_{p\in S^\RR_{\om}(X)}
\st(p(\FR)+\Lambda^\sharp),\]
where $S^\RR_{\om}(X)$ is the collection of all complete
$\CL_{\om}$-types on $X$ over $\RR$. Using Proposition~\ref{omin-prop} we obtain the
desired result.\end{proof}

\subsection{Neat families in the o-minimal context}
We proceed  similarly to Section~\ref{sec:famil-affine-subsp}.

We
denote by $\Graff_k(\RR^n)$ the Grassmannian variety of all affine
$k$-dimensional $\RR$-subspaces of $\RR^n$. For an  $\RR$-flat
$A\subseteq \RR^n$  we
denote by $L(A)$ the $\RR$-subspace  of $\RR^n$ whose translate is
$A$.

For  a subset $T\subseteq \Graff_k(\RR^n)$ we denote by $\RR\mathrm{L}(T)$ the
$\RR$-linear span of $\bigcup_{A\in T} L(A)$ in $\RR^n$.

Given an $\RR$-flat $A\subseteq \RR^n$ and an additive subgroup $\Lambda\leq \RR^n$
whose $\RR$-span equals $\RR^n$, we denote by $A^\Lambda$ the $\RR$-flat
$a+{L(A)}^\Lambda$, where $a\in A$.

\begin{defn}
A definable $T\sub \Graff_k(\RR)$
is called \emph{neat} if
\begin{enumerate}[(a)]
\item $T$ is a connected $\RR$-submanifold of $\Graff_k(T)$. \item For any
nonempty open subset $U\sub T$, we have $\RR\mathrm L(U)=\RR\mathrm L(T)$.
\end{enumerate}
\end{defn}

The notion of neatness helps us replace the irreducibility assumption in Proposition~\ref{prop:neat-dense}. We have:

\begin{prop}\label{prop:neat-dense-2}
Let $T\subseteq \Graff_k(\RR^n)$ be a definable neat family and $V=\RR
\mathrm{L}(T)$.
 Let $\Lambda< \RR^n$ be a countable additive subgroup of  $\RR^n$
 whose $\RR$-span is the whole $\RR^n$. Then,
  \begin{enumerate}
  \item The set
  $\{A\in T \colon  {L(A)}^\Lambda =V^\Lambda\}$
  is topologically dense in $T$.
\item  The set $\bigcup_{A\in T} (A+\Lambda)$ is topologically dense in
$\bigcup_{A\in
  T} (A+V+\Lambda)$ with respect to the Euclidean topology on
$\RR^n$.
\end{enumerate}
\end{prop}
\begin{proof} (1) Since $\Lambda$ is countable there are countably many spaces
${L(A)}^\Lambda$ as $A$ varies in $T$. Using Baire Categoricity it is enough to prove
that for  any proper subspace $W\lneq V^\Lambda $ which is defined over $\Lambda$,
the set $L^{-1}[W]=\{A\in T:L(A) \sub W\}$ is nowhere dense in $T$. Because this set
is definable we just need to prove that it does not contain an open set. But since
$T$ is neat, for every open $U\sub T$ we have $\RRL(U)=V$, and in particular, $U$ is
not contained in $L^{-1}[W]$.

(2) The proof is identical to that of Proposition~\ref{prop:neat-dense}(2). \end{proof}

The next result will replace the decomposition of an algebraic variety into its
irreducible components.
\begin{thm}\label{thm:comb-neat1}
Let $T\sub \Graff_k(\RR^n)$ be a definable  family of $\RR$-flats in $\RR^n$. Then $T$
can be decomposed  into a finite union of neat families.
\end{thm}
\begin{proof}
We use induction on $\dim_\RR(T)$.
If $\dim_\RR(T)=0$ then $T$ is finite and the theorem is trivial.

Assume  $\dim_\RR(T)>0$.

For  $U_1\sub U_2\subseteq T$, we have $\RRL(U_1)\sub
\RRL(U_2)$. Hence, for
$A\in T$ there exists a neighborhood $U\sub T$ of $A$ and  a subspace $V_A\sub
\RR^n$ such that for every  neighborhood  $U'\sub U$ of $A$ we have $\RR
\mathrm{L}(U')=V_A$.

The map $A\mapsto V_A$ is definable hence we may partition $T$ into finitely many
connected submanifolds $T_1,\ldots, T_r$, such that on each $T_i$ the
dimension of $V_A$ is constant. By induction, it is enough to prove that those $T_i$
of maximal dimension are neat. So we assume that $\dim T_i=\dim T$ and
$T_i$ is an open subset of $T$.

We need to show that  for each nonempty open $W\sub T_i$, we have
$\RRL(W)=\RRL(T_i)$.

We first claim that for each $A\in T_i$, there is a neighborhood $U$ of $A$
such that $V_A=V_B$ for all $B\in U$. Indeed, pick $U\sub T_i$ such that $\RR
\mathrm L(U)=V_A$. By definition, for every $B\in U$ we have $V_B\sub V_A$, but
because $\dim V_B=\dim V_A$ we must have $V_A=V_B$ for all $B\in U$.

Now,  since $T_i$ is connected, it easily follows that for all $A,B\in T_i$, we have
$V_A=V_B$.

To finish we just note that for every $A\in T_i$, we have $L(A)\sub V_A$, so
$\RRL(T_i)=V_A=\RRL(W)$, for every nonempty open $W\sub T_i$.\end{proof}

\subsection{Proof of the main theorem}

We recall the Main Theorem in the o-minimal setting:
\begin{thm}
\label{thm:main1-omin}  Let $X\subseteq \RR^n$ be a
  closed set definable in an o-minimal expansion $\RR_{\om}$ of the real field.
   There are linear $\RR$-subspaces  $V_1,\dotsc,V_m \sub
 \RR^n$ of positive dimension, and for each $i=1,\ldots, m$ a definable closed
 $C_i\sub V_i^\perp$,  such that for any lattice
 $\Lambda<\RR^n$ we have
\[  \cl(X+\Lambda) =(X+\Lambda)
    \cup\bigcup_{i=1}^m (C_i+V_i^\Lambda+\Lambda).
\]

In addition:

\begin{enumerate}[(i)]
 \item
For each $i=1,\ldots, m$ we have $\dim C_i<\dim X$

\item For each $V_i$ that is maximal among $V_1,\ldots, V_m$ the set $C_i\sub \RR^n$
is bounded, and in particular $C_i+V_i^{\Lambda}+\Lambda$ is a closed set.
\end{enumerate}
\end{thm}
\begin{proof}
As in the algebraic case  it is sufficient to find definable $C_i$'s  as above that are
that are not necessarily closed, and then
replace  each $C_i$ with its topological closure  if needed.

By Theorem~\ref{cor:omin}  for any lattice $\Lambda\sub \RR^n$ we have
\[\cl(X+\Lambda)=\bigcup_{A\in \CA^\RR(X)} (\cl(A)+\Lambda),\]
and the family $\CA^\RR(X)$ is definable.

 By Theorem~\ref{thm:comb-neat1}, applied to each $\CA^\RR(X)\cap \Graff_k(\RR)$
for $k>0$,  we
decompose  $\CA^\RR(X)\setminus X$ into finitely many neat
sets
\begin{equation}
\label{eq:Ti-omin}\CA^\RR(X)\setminus X= T_1\cup\cdots \cup T_m.
\end{equation}

We now finish the proof exactly as in Theorem~\ref{thm:main0}, with
Proposition~\ref{prop:neat-dense-2} replacing
 Proposition~\ref{prop:neat-dense}, and conclude: There are linear subspaces $V_1,\ldots, V_m\sub
\RR^n$ of positive dimension and definable sets $C_i\sub V_i^{\perp}$ such that for
any lattice $\Lambda\sub \RR^n$ we have:
\[ \cl(X+\Lambda)= (X+\Lambda)\cup \bigcup_{i=1}^r (C_i+V_i^{\Lambda}+\Lambda).\]

This ends the proof of the main statement of
Theorem~\ref{thm:main1-omin}.
We are left to prove the remaining two clauses in the theorem.

\subsubsection{Proof of Clause (i)}

We need to prove that each $C_i$ in the description of $\cl(X+\Lambda)$ has
dimension smaller than $\dim_\RR X$. For that we recall first that for each $c\in
C_i$ there exists $\alpha\in X^\sharp\smallsetminus \CO_\FR^n$ such that
$A_\alpha^\RR+V_i= c+V_i$.

We let $V=V_i$ and identify $\RR^n$ with $V^\perp \times V$. The idea of the proof
is that $C$ corresponds to the projection of the frontier of $X$ in $V^\perp \times
V^*$, where $V^*$ is the one-point compactification of $V$. By o-minimality, this
frontier will have dimension smaller than $\dim_\RR X$. We now describe the details.
Let
\[X'=\{(x_1,x_2, 1/|x_2|)\in V^{\perp} \times V\times  \RR :x_2\neq 0\, \&\, (x_1,x_2)\in X\}.\]
 Consider the projection $(x_1,x_2,1/|x_2|)\mapsto
(x_1,1/|x_2|)$ of $X^{'}$  into $V^{\perp}\times \RR$, and the image of $X'$ under
this projection,  call it $Y'$.
\\

 \noindent{\bf Claim.} For every $c\in C$,
 $(c,0)$ is in the closure of $Y'$.

\begin{proof} Assume that $c+V=A^\RR_\alpha+V\in \CF_V$ for $\alpha\in
X^\sharp\setminus \CO^n_{\RR}$. Because  $L(A^\RR_\alpha)\sub V$ it follows that
$\alpha\in c+V^\sharp+\mu_\FR^n$  and hence $\alpha$ can be written as $\alpha_1+\alpha_2$
with
 $\st(\alpha_1)=c$ and $\alpha_2\in V^\sharp$ necessarily unbounded. We thus have  and
 $\st(1/|\alpha|)=0$ and therefore
 $(c,0)$ is in $\cl(Y')$. This ends the proof of the claim. \end{proof}

 In order to finish we note that every element of the form
 $(c,0)\in V^\perp \times \RR$ is necessarily in
 the frontier $\mathrm{Fr}(Y')=\cl(Y')\setminus Y'$, thus
\[\dim_\RR C\leq \dim_\RR \mathrm{Fr}(Y') <\dim_\RR Y',\]
(the last equality follows from
o-minimality).

 But $\dim_{\RR} Y'\leq \dim_\RR X'\leq \dim_{\RR} X$, so $\dim_\RR C<\dim_\RR X$.

This ends the proof of Clause (i).

\subsubsection{Proof of Clause (ii)}
The proof below follows closely the proof of Clause (ii) in the
algebraic case.

We  start with $T=T_i$ as in \eqref{eq:Ti-omin}, and $V=V_i$, $C=C_i\sub V_i^\perp$
the corresponding linear space and definable set, respectively. Namely, $T\sub
\Graff_{k_i}(\RR)$ is neat, $V=\RRL(T)$, and
\[ C=\{ c\in V^\perp \colon A\subseteq c+V \text{ some } A\in T\}. \]
We assume that $C$ is unbounded and we  show that there exists $\alpha^*\in X^\sharp$
such that $L(A^\RR_\alpha)$ properly contains $V$. This implies that $V$ cannot be
maximal among the $V_j$'s.

Let $\Sigma$ be the collection of all proper $\RR$-subspaces of $\RR^n$ which do not
contain $V$.

As in the algebraic case,   if no such $\alpha^*$  exists then the intersection
\[
X^\sharp\cap \{ u\in  \FR^n \colon u\notin (V^\sharp+\CO_\FR^n) \} \cap\bigcap_{W\in
\Sigma} \{ u\in  \FR^n  \colon u \notin (W^\sharp+\CO_\FR^n) \}
\]
would be  empty, and we would conclude that there is a closed ball $B\sub \RR^n$ and
$W_1,\ldots, W_l\in \Sigma$ such that for every $A\in \CA_\RR(X)$,
\begin{enumerate}
\item either $A\subseteq V+B$; \item  or  $L(A) \subseteq W_i$.
\end{enumerate}

For every $A\in T$ we have $L(A)\sub V$, hence $L(A)\sub W_i$ implies that $L(A)\sub
W_i\cap V$. Because $T$ is neat and each $W_i\cap V$ is a proper subspace of $V$ the
dimension of the set $T'=\bigcup_{i=1}^m \{A \in T: L(A)\sub W_i\}$  is smaller than
$\dim_{\RR} T$. So, for all $A\in T$ outside a set $T'$ of smaller dimension, we
have $A\subseteq V+B$. Let us see that this is impossible.

For $A\in T$, we denote by $c(A)\in V^\perp$  the unique $c\in V^\perp$ so that
$A\sub c+V$.  The map $A\mapsto c(A)$ is continuous and surjective so the pre-image
of $C\setminus B$ is a non-empty, open subset of $T$. Since $T$ is a connected
manifold the intersection of this pre-image with $T\setminus T'$ is non-empty, so
there exists $A\in T\setminus T'$ with $A\notin V+B$. Contradiction. This ends the
proof of Clause (ii) and  with that we end the proof of
Theorem~\ref{thm:main1-omin}.\end{proof}

\section{An example}
\label{sec:examples}

In this section we provide a counter-example to conjectures 
\cite[Conjecture 1.2]{UY} and \cite[Conjecture 1.6]{flow}.

Let $X$
be  the surface
\[  X=\{ (x,y,z)\in \CC^3 \colon  x(1-yz)=1\} \]
and $\Lambda=\ZZ^3+\mathrm{i}\ZZ^3$.

For $i=1,2,3$ we denote by $\Pi_i$ the corresponding coordinate plane in $\CC^3$,
namely $\Pi_1=0{\times}\CC{\times} \CC$,
 $\Pi_2=\CC{\times}0{\times} \CC$ and
$\Pi_3=\CC{\times}\CC{\times} 0$.  Notice that all $\Pi_i$ are defined
over $\Lambda$, hence each $\Pi_i+\Lambda$ is closed.

\begin{lem}
  \begin{equation}
\label{eq:5}
\cl(X+\Lambda)=
(X+\Lambda)\cup\bigcup_{i=1}^3
(\Pi_i+\Lambda)
\cup
\bigcup_{t\in \CC^*} \bigl((0,t,\tfrac{1}{t})+\CC{\times} 0 {\times} 0 +\Lambda\bigr).
  \end{equation}
\end{lem}

\begin{proof}
We use the formula
\[ \cl(X+\Lambda)= (X+\Lambda)\cup \bigcup_{A\in \CA^*(X)} \cl(A+\Lambda).\]

\medskip
\noindent\textbf{Inclusion $\mathbf\supseteq$.}
We first show that $\cl(X+\Lambda)$ contains
every set in the union on the right side of \eqref{eq:5}.

Obviously $\cl(X+\Lambda)$  contains $X+\Lambda$.

For $\Pi_1+\Lambda$, consider $\alpha=(\frac{1}{1-\delta^3}, \delta,\delta^2)$,
with nonzero $\varepsilon\in
\mu_\FC$ and  $\delta=1/\varepsilon$.
We show that $\Pi_1=A_\alpha^\CC(X)$.
Clearly $\delta\notin
\CO_\FC$ and   $\alpha\in
X^\sharp \setminus \CO^3_\FC$.  Since $\frac{1}{1-\delta^3}\in \mu_\FC$, we have
$\alpha \in \Pi_1^\sharp+\mu_\FC^3$.  Also,  $\delta$
and $\delta^2$ are $\CC$-linearly independent modulo $\CO_\FC$ (i.e.\ for $c_1,c_2\in
\CC$, if  $c_1\delta + c_2\delta^2\in \CO_\FC$ then $c_1=c_2=0$),
hence   $\alpha \notin
L^\sharp+\CO^3_\FC$ for any proper $\CC$-subspace $L\subsetneq  \Pi_1$.  Thus
$A^\CC_\alpha=\Pi_1$ and, by Theorem~\ref{thm:main-1},
\[ \Pi_1+\Lambda \subseteq \cl(X+\Lambda). \]

For $\Pi_2+\Lambda$,
consider   $\alpha=(\delta^2, \varepsilon -\varepsilon^3, \delta)$.
We have
that $\alpha\in X^\sharp \setminus \CO^3_\FC$  and as above it is easy to see that
$A_\alpha^\CC= \Pi_2$. Hence
\[ \Pi_2+\Lambda \subseteq \cl(X+\Lambda). \]

Similarly, for  $\alpha=(\delta^2,\delta, \varepsilon -\varepsilon^3)$,  we have
$A^\CC_\alpha= \Pi_3$, Hence
\[ \Pi_3+\Lambda \subseteq \cl(X+\Lambda). \]

For $t\in \CC^*$ and $\epsilon\in \mu_\FC$, let
$\alpha_t=(\frac{t}{\varepsilon},t-\varepsilon,\frac{1}{t})$. It is easy to see that
$\alpha_t\in X^\sharp \setminus \CO_\FC^3$ and $A_{\alpha_t}^\CC =
(0,t,\frac{1}{t})+\CC{\times}0{\times}0$.

Thus the right side of \eqref{eq:5} is contained in $\cl(X+\Lambda)$.

\medskip
\noindent\textbf{Inclusion $\mathbf \subseteq$.}
To show that $\cl(X+\Lambda)$ is contained in the right side of \eqref{eq:5}, it is
sufficient to show that for any $\alpha\in X^\sharp\setminus \CO^3_\FC$ the set
$\cl(A^\CC_\alpha+\Lambda)$ is contained in the right side of \eqref{eq:5}.

Let $\alpha=(\alpha_1,\alpha_2,\alpha_3)\in X^\sharp\setminus
\CO^3_\FC$.

 Observe that if some $\alpha_i\in \mu_\FC$
then $\alpha\in \Pi_i^\sharp+\mu_\FC$. In this case $A^\CC_\alpha \subseteq \Pi_i$
and $\cl(A^\CC_\alpha+\Lambda)\subseteq \Pi_i+\Lambda$, hence it is contained in the
right side.

The only remaining case is when $\alpha\in X^\sharp$ is unbounded and
none of $\alpha_i$ is in $\mu_\FC$.  It is easy to see that in this
case $\alpha_1$ must be unbounded, $\alpha_2\in t+\mu_\FC$, and
$\alpha_3\in \frac{1}{t}+\mu_\FC$ for some $t\in \CC^*$.  Then
$A_\alpha= (0,t,\frac{1}{t})+\CC{\times}0{\times}0$,
$\cl(A_\alpha+\Lambda)=
(0,t,\frac{1}{t})+\CC{\times}0{\times}0+\Lambda$, and it is contained
in the right side of \eqref{eq:5}.
\end{proof}

Thus in the  notations of the main theorem we can write $\cl(X+\Lambda)$ as
\[ \cl(X+\Lambda)=
(X+\Lambda)\cup\bigcup_{i=1}^4(C_i+V_i+\Lambda), \] where $V_i=\Pi_i$ for
$i=1,\dotsc,3$, $C_1=C_2=C_3= 0{\times 0}{\times 0}$, $V_4=\CC{\times} 0 {\times}
0$, and $C_4=\{(x,y,z)\in \CC^3 \colon
  x=0\, ,yz=1\}$.

\medskip
 Consider the projection map $\pi:\CC^3\to E^3$, where
$E$ is the elliptic curve $\CC/(\ZZ+i\ZZ)$. The set $\pi(C_4+V_4)$ is just
$\pi(C_4)+E{\times} 0{\times} 0$ and it is not hard to see that this
set is not contained in any proper real analytic subvariety of $E^3$.

Thus  $\cl(\pi(X))$ cannot be written as the union of $\pi(X)$
with a real analytic subvariety of $E^3$. Because $X$ is  semialgebraic, this
shows the failure of the  Conjecture  1.6 from \cite{flow}.

For the same reason   $\cl(\pi(X))$ cannot
be written as the union of $\pi(X)$ with finitely many real weakly special
subvarieties of $E^3$. This shows also the failure of  Conjecture 1.2
from \cite{UY}.

\bibliographystyle{acm}
\bibliography{flows-paper-v2}

\end{document}